\documentclass[10pt]{amsart}

\usepackage{hyperref}
\usepackage{mathrsfs}
\usepackage{amssymb}
\usepackage{latexsym}

\usepackage[mathscr]{eucal}
\usepackage{slashed}

\usepackage[all]{xy}
\usepackage{pb-diagram,pb-xy}
\usepackage[a4paper, hmargin=3.0cm, tmargin=3.2cm, bmargin=3.2cm]{geometry}
\numberwithin{equation}{section}

\newtheorem{MainThm}{Theorem}

\theoremstyle{definition}

\newtheorem{definition}[equation]{Definition}
\newtheorem{definition-proposition}[equation]{Definition/Proposition}
\newtheorem{assumptions}[equation]{Assumption}
\newtheorem{observation}[equation]{Observation}

\theoremstyle{remark}
\newtheorem{remark}[equation]{Remark}
\newtheorem{example}[equation]{Example}

\theoremstyle{plain}
\newtheorem{theorem}[equation]{Theorem}
\newtheorem{lemma}[equation]{Lemma}
\newtheorem{proposition}[equation]{Proposition}

\newtheorem{corollary}[equation]{Corollary}
\newcommand{\bI}{\mathbb{I}}
\newcommand{\bC}{\mathbb{C}}

\newcommand{\bN}{\mathbb{N}}

\newcommand{\bR}{\mathbb{R}}

\newcommand{\bZ}{\mathbb{Z}}
\newcommand{\bH}{\mathbb{H}}

\newcommand{\cG}{\mathcal{G}}
\newcommand{\cS}{\mathcal{S}}
\newcommand{\cW}{\mathcal{W}}

\newcommand{\cH}{\mathcal{H}}

\newcommand{\cP}{\mathcal{P}}
\newcommand{\cU}{\mathcal{U}}

\newcommand{\fX}{\mathfrak{X}}

\newcommand{\fF}{\mathfrak{F}}

\newcommand{\fM}{\mathfrak{M}}
\newcommand{\fL}{\mathfrak{L}}

\newcommand{\Bott}{\beta}
\newcommand{\bott}{\operatorname{bott}}
\newcommand{\Cal}{\mathbf{Cal}}
\newcommand{\Lin}{\mathbf{Lin}}
\newcommand{\Fin}{\mathbf{Fin}}

\newcommand{\Cl}{\mathbf{Cl}}

\newcommand{\Diff}{\mathrm{Diff}}

\newcommand{\dom}{\mathrm{dom}}
\newcommand{\eps}{\epsilon}

\newcommand{\End}{\operatorname{End}}
\newcommand{\Fred}{\operatorname{Fred}}

\newcommand{\grotimes}{\hat{\otimes}}
\newcommand{\Hom}{\operatorname{Hom}}

\newcommand{\ind}{\operatorname{index}}
\newcommand{\inddiffH}{\operatorname{inddiff}_H}
\newcommand{\inddiffGL}{\operatorname{inddiff}_{GL}}
\newcommand{\id}{\operatorname{id}}

\newcommand{\Kom}{\mathbf{Kom}}
\newcommand{\map}{\operatorname{map}}
\newcommand{\mor}{\operatorname{Mor}}
\newcommand{\normalize}[1]{\frac{#1}{(1+{#1}^{2})^{1/2}}}

\newcommand{\pr}{\operatorname{pr}}
\newcommand{\pseudo}[1]{\Psi \mathrm{DO}^{#1}}
\newcommand{\pseudop}[2]{\Psi \mathrm{DO}_{#1}^{#2}}
\newcommand{\Pseudop}[2]{\overline{\Psi \mathrm{DO}}_{#1}^{#2}}
\newcommand{\pseudir}[1]{{\Psi \mathrm{Dir}}(#1)}

\newcommand{\Lopspace}{\mathrm{map}_{C^{\infty}} ((\mathbb{I}, \partial \mathbb{I}, -1 ); (\Psi \mathrm{Dir} (V), \Psi \mathrm{Dir} (V)^{\times}, B)) }

\newcommand{\spfl}{\mathrm{sf}}
\newcommand{\susp}{\operatorname{susp}}
\newcommand{\spec}{\operatorname{spec}}

\newcommand{\supp}{\operatorname{supp}}
\newcommand{\Sym}{\mathrm{Sym}}
\newcommand{\Symb}{\mathrm{Smb}}
\newcommand{\symb}{\operatorname{smb}}
\newcommand{\twomatrix}[4]{\begin{pmatrix} #1 & #2 \\ #3 & #4  \end{pmatrix}}

\newcommand{\cstar}{\mathrm{C}^{\ast}}

\newcommand{\Dir}{\slashed{\mathfrak{D}}}
\newcommand{\spinor}{\slashed{\mathfrak{S}}}
\newcommand{\Riem}{{\mathcal R}}

\newcommand{\gR}{\mathbf{R}}

\newcommand{\scal}{\mathrm{scal}}

\newcommand{\norm}[1]{\| #1 \|}

\newcommand{\KO}{KO}

\newcommand{\KK}{KK}

\newcommand{\vol}{\mathrm{vol}}

\title{The two definitions of the index difference}

\author{Johannes Ebert}

\address{Mathematisches Institut, Universit\"at M\"unster\\
Einsteinstra{\ss}e 62\\
48149 M\"unster\\
Bundesrepublik Deutschland}

\email{johannes.ebert@uni-muenster.de}
\date{\today}



\begin{document}

\begin{abstract}
Given two metrics of positive scalar curvature on a closed spin manifold, there is a secondary index invariant in real $K$-theory. There exist two definitions of this invariant, one of homotopical flavor, the other one defined by an index problem of Atiyah-Patodi-Singer type. We give a complete and detailed proof of the folklore result that both constructions yield the same answer. Moreover, we generalize this result to the case of two families of positive scalar curvature metrics, parametrized by a finite CW complex.
In essence, we prove a generalization of the classical ``spectral-flow-index theorem'' to the case of families of real operators.
\end{abstract}

\maketitle

\section{Introduction}

\subsection{Two secondary index invariants for positive scalar curvature metrics}

The \emph{space of metrics of positive scalar curvature} $\Riem^+ (M)$ on a closed spin manifold $M$ of dimension $d$ has attracted the attention of homotopically minded geometric topologists. One tool for its study comes from index theory. 
Given any Riemannian metric $g$ on $M$, one can consider the Atiyah-Singer-Dirac operator $\Dir_g$, which is an odd, $\Cl^{d,0}$-linear and symmetric elliptic operator that acts on the space $H=L^2 (M ; \spinor_M)$ of sections of the real spinor bundle $\spinor_M \to M$. Thus it yields a point in $\Fred^{d,0}(H)$, the space of self-adjoint odd $\Cl^{d,0}$-linear Fredholm operators. By a classical result of Atiyah-Singer and Karoubi, $\Fred^{d,0}(H)$ represents the real $K$-theory functor $\KO^{-d}$, and thus we get an index $\ind_{d,0} (\Dir_g) \in \KO^{-d}(*)$. If $g$ has positive scalar curvature, then $\ind_{d,0} \Dir_g=0$: the well-known argument using the Schr\"odinger-Lichnerowicz formula proves that $\Dir_g$ is invertible, which shows that at first glance, the index does not convey any information.
There are, however, useful \emph{secondary} index-theoretic invariants that one can attach not to single metric of positive scalar curvature, but to a \emph{pair} $(g_{-1},g_1)$ of such. There are two constructions of such invariants. What they have in common is that one starts by considering the family $g_t= \frac{1-t}{2} g_{-1}+\frac{1+t}{2}g_1$ of metrics, parametrized by $[-1,1]$. Of course, and this is the point, the metric $g_t$ typically does not have positive scalar curvature. 

For the \emph{first construction}, we consider the path $\Dir_{g_t}$ in the space $\Fred^{d,0}(H)$ (in this introduction, we are glossing over the detail that $\spinor_M$ and hence $H$ depends on $g$). This path begins and ends at an invertible operator, since $g_{\pm 1}$ has positive scalar curvature. As the space of invertible operators is contractible (Kuiper's theorem), the path contains the homotopical information of a point in the loop space $\Omega \Fred^{d,0}$. This space represents the functor $\KO^{-d-1}$, and in this way we obtain the first version of the \emph{index difference} $\inddiffH (g_{-1},g_1) \in \KO^{-d-1}(*)$. This viewpoint was introduced by Hitchin \cite{Hit}.

For the \emph{second construction}, we extend the metric $dt^2 + g_t$ on $M \times [-1,1]$ to a metric $h$ on the cylinder $M \times \bR$ (constant outside $[-1,1]$) and consider the Dirac operator $\Dir_h$ on $M \times \bR$ of this metric. As the dimension of $M \times \bR$ is $d+1$, this is a $\Cl^{d+1,0}$-linear operator. It defines a Fredholm operator on $\cH:=L^2 (M \times \bR; \spinor_{M \times \bR})$, and we can consider it as a point in $\Fred^{d+1,0}(\cH)$, giving another element $\inddiffGL \in \KO^{-d-1}(*)$ (one could view this an index of an Atiyah-Patodi-Singer boundary value problem as in \cite{APS1}, but we use a different setup). This viewpoint on the index difference is due to Gromov and Lawson \cite{GL}.

An obvious question is whether both constructions yield the same result and this question is most cleanly formulated in the family case, to which both constructions can be generalized. There are well-defined homotopy classes of maps

\begin{equation}
\begin{split}
\inddiffH: \Riem^+ (M) \times \Riem^+ (M) \to \Omega \Fred^{d,0} (H), \\
 \inddiffGL : \Riem^+ (M) \times \Riem^+ (M) \to \Fred^{d+1,0} (H).
\end{split}
\end{equation}
Both map the diagonal to the (contractible) space of invertible operators. The Bott periodicity map is a weak equivalence 
\[
\bott: \Fred^{d+1,0} \stackrel{\simeq}{\to} \Omega \Fred^{d,0}.
\]

\begin{MainThm}\label{mainresult-psc}
The two maps $\bott \circ \inddiffGL$ and $\inddiffH$ are weakly homotopic. In other words, if $X$ is a finite CW complex and $f: X \to \Riem^+ (M) \times \Riem^+ (M)$ is a map, then $\bott \circ \inddiffGL \circ f$ and $\inddiffH \circ f$ are homotopic. Moreover, if the subcomplex $Y \subset X$ is mapped to the diagonal, the homotopy can be chosen to be through invertible operators on $Y$.
\end{MainThm}

The restriction to finite CW pairs is out of convenience; it can probably be removed, but we opine that this is not worth the effort. Certainly, Theorem \ref{mainresult-psc} does not come as a surprise at all and in fact it has the status of a folklore result. However, we are not aware of a published adequate exposition. 
As indicated, the secondary index is an important tool to detect homotopy classes in the space $\Riem^+ (M)$, compare \cite{GL,Hit,HSS,CS}. The authors of these works have been careful to avoid any use of Theorem \ref{mainresult-psc}. In \cite{BERW}, a stronger detection theorem for $\pi_* (\Riem^+ (M))$ is proven, for manifolds of all dimensions $\dim (M) \geq 6$. In that paper, Theorem \ref{mainresult-psc} is used in an essential way to pass from $2n$-manifolds to $(2n+1)$-dimensional manifolds. Thus, filling this gap was motivated not by an encyclopedic striving for completeness, but by necessity. 

\subsection{Spectral-flow-index theorem}

Let us explain the strategy for the proof of Theorem \ref{mainresult-psc}, first under the assumption that $X=*$ and $d +1  \equiv 0 \pmod 8$. In this case, $\pi_0 (\Fred^{d+1,0}(H))$ is isomorphic to $\bZ$, detected by an ordinary Fredholm index. The space $\Fred^{d,0}(H)$ is homeomorphic to the space of self-adjoint Fredholm operators on a real Hilbert space. The fundamental group $\pi_1 (\Fred^{d,0}(H))$ is isomorphic to $\bZ$, and it is detected by the ``spectral flow'' $\spfl$: if $A$ is a family of operators such that $A(\pm 1 )$ invertible, then $\spfl (A)$ is the number of eigenvalues of $A(t)$ that cross $0$, counted with multiplicities. This concept was introduced by Atiyah-Patodi-Singer \cite[\S 7]{APS3}. To such a family, one can associate the operator $D_A:= \partial_t + A(t)$, acting on $L^2 (\bR;H)$. This is Fredholm, and so has an index. The ``spectral-flow index theorem'' states that $\spfl (A) = \pm \ind (D_A)$ (for the moment, we ignore the sign). The basic idea why Theorem \ref{mainresult-psc} is expected to be true is that Hitchin's index difference is given by the spectral flow of the family $A(t)=\Dir_{g_t}$, while Gromov-Lawson's index difference is the index of the operator $D_{A}$. 

There are many approaches to spectral-flow-index theorems in the literature. 
Atiyah-Patodi-Singer gave a proof for a special case \cite[Theorem 7.4]{APS3}: $A(t)$ has to be an elliptic operator on a closed manifold, and the crucial assumption is that $A(1)=A(-1)$. This is important, because their proof is by gluing the ends together in order to reduce to an index problem on $M \times S^1$, which can be solved by the usual Atiyah-Singer index theorem. However, this assumption is not satisfied in our case. 
One obvious first idea for such a reduction would be to use that $A(1)$ and $A(-1)$ are homotopic through invertible operators, by Kuiper's theorem. Composing the family $A(t)$ with such a homotopy would result in a closed family $A'$, and the operator $D_A'$ would have the same index. But the index problem for the new family still cannot be reduced to a problem on $M \times S^1$, for a very fundamental reason: for that to work the homotopy must be through \emph{pseudo-differential operators}. Even though $A(1)$ and $A(-1)$ are homotopic through elliptic differential operators and through invertible operators on the Hilbert space, we cannot fulfill both requirements at once! In fact, our proof will clearly show that this is the essential information captured by the spectral flow. 

Before we describe our argument, let us discuss several other approaches that appeared in the literature. Bunke \cite{Bu93} considers the case when $A(t)$ is a family of differential operators with the same symbol. He reduces the problem to the closed case; but his answer is in terms of cohomology, and he does not treat the parametrized situation. Translating his argument to real $K$-theory would, as far as we can see, not have resulted in a shorter proof of Theorem \ref{mainresult-psc}. 
Robbin and Salamon \cite{RoSa} worked in an abstract functional analytic setting, ignoring that the operators are pseudo-differential. For the case $X=*$, $d+1 \equiv 0 \pmod 8$, they gave a detailed proof in this abstract setting. We will use a special case of their result in our proof; but we failed with an attempt at a straightforward generalization of their argument to the family case.
Another proof in the framework of $\KK$-theory is due to Kaad and Lesch \cite{KL}, again the details are only for the complex case and $X=*$; our knowledge of Kasparov theory does not suffice to carry out the generalization to the case we need.

\subsection{Overview of the paper}

Let us now give a description of what we actually do. Chapter \ref{k-survey} surveys background material on Clifford algebras and $K$-Theory. In section \ref{clifford-survey}, we collect the conventions on Clifford algebras that we use (we use all Clifford algebras $\Cl^{p,q}$ to make the linear algebra work better). Section \ref{fredholmsec1} recalls the classical Fredholm model for real $K$-Theory. We have to generalize the Fredholm model in such a way that a Hilbert bundle, together with a Fredholm family represents an element in $K$-Theory. There are well-known difficulties with the structural group and the continuity condition on a Fredholm family. Therefore, we spend some pages explaining these conditions (section \ref{hilbertbundles}). Section \ref{fredholmsec2} discusses the generalized Fredholm model; the proof that the construction gives the correct answer is deferred to the appendix \ref{appendix}.
A side-purpose of chapter \ref{k-survey} is to close a gap in the literature. Classically, the family index theorem is only formulated for compact base spaces, and even a definition of the family index over a noncompact base does not seem to be discussed properly in the literature. In \cite{BERW}, we need to consider family indices over not even locally compact bases, and chapter \ref{k-survey}, together with the appendix, was partially written with that goal in mind.

The goal of chapter \ref{analysis} is to give the rigorous definition of the secondary index invariants and the formulation of the main result of this paper (Theorem \ref{main-index-theorem}). Section \ref{elliptic-preliminaries} collects some facts on elliptic regularity for manifolds with cylindrical ends; the key result is Proposition \ref{fredholmproperty:familywise}. In order to cover the other index-theoretic arguments that appear in \cite{BERW}, we prove more general versions than necessary for Theorem \ref{mainresult-psc}. 
The overall structure of the proof of Theorem \ref{mainresult-psc} forces us to leave the realm of Dirac operators; we have to work with pseudo-differential operators that have the leading symbol of a Dirac operator (we call them ``pseudo-Dirac operators''). The general setting for our index theorem are families $A(t)$ of $\Cl^{p,q}$-linear pseudo-Dirac operators on a closed manifolds. Given such a family, we get a new operator $D_A$ on the manifold $M \times \bR$, which is $\Cl^{p+1,q}$-linear (called \emph{suspension}).
While the family $A(t)$ yields $\inddiffH$, the operator $D_A$ corresponds to $\inddiffGL$. We organize the curves of Dirac operators on $M$, parametrized by a space $X$, in a suitable $K$-group that we call $L^{p,q}(X)$. The two constructions (family index of $A(t)$ and family index of $D_A$) give maps $L^{p,q}(X) \to \KO^{q-p-1}(X)$ and our main result (Theorem \ref{main-index-theorem}) says that both maps are equal.
In section \ref{defnsinddiff}, we deduce Theorem \ref{mainresult-psc} from Theorem \ref{main-index-theorem}.

Chapter \ref{ktheory} contains the proof of Theorem \ref{main-index-theorem} and we follow a common strategy for proving index theorems.  The crucial analytical ingredient for the proof of Theorem \ref{main-index-theorem} is 
Proposition \ref{thm:invertiblepseudos} which states that the space of curves of $\Cl^{p,q}$-linear pseudo-Dirac operators $A(t)$, $A(\pm 1)$ invertible, is rich enough to realize $\Omega \Fred^{p,q}$. This is inspired by a theorem of Boo{\ss}-Wojciechowski \cite{BW} and would be false if we tried to use differential operators. In section \ref{formal-structures}, we use Proposition \ref{thm:invertiblepseudos} and formal properties of $K$-theory to reduce everything to the special case $X=*$ and  $(p,q)=(0,1)$, which is the case that was dealt with by Robbin and Salamon (in fact, we use an explicit index computation instead).

\subsection*{Acknowledgements}

I am grateful to Boris Botvinnik and Oscar Randal-Williams for the exciting collaboration (of which the present paper is an outsourced part); moreover I want to thank Boris and his wife Irina for the warm welcome in their home.
Thomas Schick made several useful comments on this paper, but in particular I want to thank the anonymous referee for reading the paper in an extremely careful way.

\section{Preliminaries on \texorpdfstring{$K$}--Theory and Clifford algebras}\label{k-survey}

\subsection{Clifford algebra}\label{clifford-survey}

Throughout the paper, we work over the real numbers; the proofs can easily be ``complexified''.
Without any further mentioning, we assume (or claim implicitly) that all Hilbert spaces are \emph{separable}.

\begin{definition}
Let $V$ and $W$ be two finite-dimensional euclidean vector spaces and let $H$ be a Hilbert space. A \emph{$\Cl^{V,W}$-structure} on $H$ is a pair $(\iota,c)$, where $\iota$ is a self-adjoint involution of $H$ and $c: V \oplus W \to \Lin (H)$ is a linear map to the space of bounded operators on $H$ such that 
\[
c(v,w) \iota = -\iota c(v,w), \; c(v,w)^* = c (-v,w), \; c(v,w)^2=-|v|^2 + |w|^2 
\]
for all $v \in V$, $w \in W$. A \emph{$\Cl^{V,W}$-Hilbert space} is a Hilbert space, equipped with a $\Cl^{V,W}$-structure. The \emph{opposite} $\Cl^{V,W}$-Hilbert space is $(H,\iota,c)^{op}:=(H,-\iota,-c)$ and is shortly denoted by $H^{op}$.
A bounded linear operator $F: (H,\iota,c) \to (H',\iota',c')$ of $\Cl^{V,W}$-Hilbert spaces will be called \emph{Clifford-linear} if $Fc(x)=c(x)F$ holds all $x \in V \oplus W$. A Clifford-linear bounded operator $F$ is \emph{even} if $F \iota = \iota' F$, and \emph{odd} if $F \iota = -\iota' F$.
\end{definition}

If there is no risk of confusion, we write $x$ for $c(x)$. Of particular interest to us is the case $V=\bR^p$, $W=\bR^q$, both with the standard scalar product. In this case, we write $\Cl^{p,q}$ instead of $\Cl^{\bR^p, \bR^q}$. A $\Cl^{p,q}$-structure on $H$ is given by orthogonal automorphisms $\iota,e_1,\ldots,e_{p}, \eps_1, \ldots , \eps_q$ of $H$, satisfying the relations 
\begin{equation}\label{clifford-relations}
\iota^2 =1, \; e_i \eps_j +  \eps_j e_i = e_i \iota + \iota e_i = \eps_j \iota +   \iota \eps_j =0, \;-(e_i e_j + e_j e_i )= \eps_i \eps_j + \eps_j \eps_i =2\delta_{ij}.
\end{equation}

There are various functors between the categories of $\Cl^{V,W}$-Hilbert spaces for different values of $(V,W)$, the classical \emph{Morita equivalences}. 
If $(H_0,c_0,\iota_0)$ is a $\Cl^{V,W}$-Hilbert space, we obtain a $\Cl^{V\oplus \bR,W \oplus \bR}$-Hilbert space $(H,\iota,c)$, $H:=H_0 \oplus H_0$ and
\[
c:=\twomatrix{c_0}{}{}{-c_0}, \; \iota:=\twomatrix{\iota_0}{}{}{-\iota_0}, \; c(e):= \twomatrix{}{-1}{1}{}, \; c(\eps):=\twomatrix{}{1}{1}{}
\]
(here $e$ and $\eps$ are the standard basis vectors of $\bR\oplus \bR$). If $F_0$ is a Clifford-linear endomorphism of $H_0$, then $F = \twomatrix{F_0}{}{}{F_0}$ is a Clifford-linear endomorphism of $H $, and if $F_0$ is even or odd, then so is $F$. Vice versa, if $H$ is a $\Cl^{V\oplus \bR,W \oplus \bR}$-Hilbert space with $e$ and $\eps$ being the actions of the basis vectors in the $\bR$-summands, consider $H_0:=\ker(\eps e-1)$; this inherits a $\Cl^{V,W}$-structure from the given one on $H$. Clifford-linear endomorphisms of $H$ restricts to Clifford-linear endomorphisms of $H_0$, and if $F$ is even or odd, then so is $F|_{H_0}$. Both procedures are mutually inverse.
In a similar fashion, $\Cl^{V \oplus \bR^4,W}$-Hilbert spaces and $\Cl^{V,W \oplus \bR^4}$-Hilbert spaces are equivalent. Let $H$ be a $\Cl^{V \oplus \bR^4,W}$-Hilbert space and put $\eta:= e_{1} \cdots e_{4}$, so that $\eta^2 =1$. Then we obtain a $\Cl^{V ,W\oplus \bR^4}$-Hilbert space $(H,\iota,c')$; $c'|_{V\oplus W} = c|_{V \oplus W}$, $\eps_{i}:= \eta e_{i}$ for $i=1,\ldots,4$.
Clifford-linear, even (or odd) endomorphisms are preserved under this procedure. By a similar recipe one transforms $\Cl^{V , W \oplus \bR^4}$-Hilbert spaces back into $\Cl^{V \oplus \bR^4, W }$-Hilbert spaces. Combining both types of equivalences, one gets equivalences between $\Cl^{V,W}$-, $\Cl^{V \oplus \bR^8,W}$- and $\Cl^{V,W \oplus \bR^8}$-Hilbert spaces.
All these definitions and constructions generalize without effort to nontrivial Riemannian vector bundles $V$, $W$ and Hilbert bundles $H$. 

The structure theory of $\Cl^{p,q}$-Hilbert spaces is well-known \cite[\S I.5]{LM}. For brevity, we call finite-dimensional $\Cl^{p,q}$-Hilbert spaces just \emph{$\Cl^{p,q}$-modules}. Each $\Cl^{p,q}$-Hilbert space decomposes into a Hilbert sum of (finite-dimensional) irreducible ones. If $p-q \not \equiv 0 \bmod 4$, then there is exactly one irreducible $\Cl^{p,q}$-module, up to isomorphism. If $p-q \equiv 0 \pmod 4$, there are exactly two irreducible $\Cl^{p,q}$-modules, up to isomorphism. These are mutually opposite and distinguished by their \emph{chirality}: consider the operator $\omega=\omega_{p,q} := \iota  \eps_1 \cdots \eps_q\cdot e_1 \cdots e_p$. If $p-q \equiv 0 \pmod 4$, then $\omega$ is $\Cl^{p,q}$-linear, even and satisfies $\omega^2=1$. In an irreducible $\Cl^{p,q}$-module, $\omega$ must act by $\pm 1$ (Schur's lemma), and this sign is the chirality.

\begin{definition}
A $\Cl^{p,q}$-Hilbert space $H$ is called \emph{ample} if it contains each irreducible $\Cl^{p,q}$-module with infinite multiplicity.
\end{definition}

Note that two ample $\Cl^{p,q}$-Hilbert spaces are isomorphic. If $H$ is an ample $\Cl^{p,q}$-Hilbert space, $s \geq p$, $t \geq q$, one can extend the $\Cl^{p,q}$-structure on $H$ to an ample $\Cl^{s,t}$-structure. There are strong homotopy theoretic versions of these statements. Let $H$ be a real Hilbert space and let $\cS^{p,q}(H)$ be the set of ample $\Cl^{p,q}$-structures; a point in $\cS^{p,q}(H)$ is given by $p+q+1$ linear isometries and we topologize $\cS^{p,q}(H)$ as a subspace of $\Lin (H)^{p+q+1}$, equipped with the norm topology. If $H$ is understood, we write $\cS^{p,q} := \cS^{p,q}(H)$. The following result is well-hidden in \cite{AS69} and \cite{Karoubi-paper}, and we make the proof explicit.

\begin{lemma}\label{kuipertheoremb}
If $H$ is infinite dimensional, the space $\cS^{p,q}=\cS^{p,q} (H)$ is contractible. Moreover, if $s \geq p$, $t \geq q$, then the forgetful map $\cS^{s,t} \to \cS^{p,q}$ is a Serre fibration with contractible fibers.
\end{lemma}
\begin{proof}
Fix $c_{s,t}\in\cS^{s,t}$ and let $c_{p,q}\in\cS^{p,q}$ be obtained by restriction. Let $U=U(H)$ be the group of isometries of $H$, equipped with the norm topology. This group acts by conjugation on $\cS^{s,t}$, and the action is transitive, by the structure theory of Clifford modules. Let $U_{p,q} \subset U$ be the stabilizer of $c_{p,q}$. This is the group of all $\Cl^{p,q}$-linear and even isometries.
By Kuiper's theorem \cite{Kui}, $U \simeq *$. By Morita equivalences and ampleness, $U_{p,q}$ is homeomorphic to either the group of isometries of an infinite-dimensional Hilbert space over $\bR, \bC$ or $\bH$ or a product of such. Thus $U_{p,q} \simeq *$ by \cite{Kui} as well.
The space $\cS^{p,q}$ is the homogeneous space $U/U_{p,q}$ and so is contractible. The forgetful map can be identified with the $U_{p,q}/U_{s,t}$-bundle $U/U_{s,t} \to U/U_{p,q}$ and so the proof is complete.
To make this argument valid, it remains to be shown that $U \to \cS^{p,q}$ has local sections (so that $U \to \cS^{p,q}$ is a $U_{p,q}$-principal bundle).
For this, use that a $\Cl^{p,q}$-structure can be viewed as an orthogonal representation of the finite group $G$ generated by symbols $e_i, \eps_j, \iota$, subject to the Clifford relations and invoke the following general lemma (take $V$ to be the sum of all irreducible $\Cl^{p,q}$-modules).
\end{proof}

\begin{lemma}
Let $G$ be a finite group and let $H$ be a Hilbert space. Let $\rho: G \to U(H)$ be a unitary representation. Let $X \subset \Hom (G,U(H))$ be the space of representations $\sigma$ such that $H_{\sigma}$  (i.e., $H$ equipped with the $G$-action induced by $\sigma$) is $G$-isomorphic to $H_{\rho}$. Endow $X$ with the norm topology, as a subspace of $U(H)^G$. Then the map $\pi: U(H) \to X$, $u \mapsto u \rho u^{-1}$, has local sections.
\end{lemma}
\begin{proof}
(This proof was suggested to us by the referee) We first construct the local section around the basepoint $\rho$. 
We need to find a neighborhood $Y \subset X$ of $\rho$ and a map $Y \to U(H)$, $\sigma \mapsto B_{\sigma}$ such that $B_{\sigma}: H_{\rho} \to H_{\sigma}$ is $G$-equivariant.
For each $\sigma \in X$, define the bounded operator
\[
A_{\sigma}:= \frac{1}{|G|} \sum_{g \in G} \sigma (g ) \rho (g^{-1}) \in \Lin (H).
\]
It is clear that $A_{\rho}=1$ and that $A_{\sigma} \circ \rho (h) = \sigma (h) \circ A_{\sigma}$.
Moreover 
\[
\norm{A_{\sigma_0} - A_{\sigma_1}} \leq \frac{1}{|G|} \sum_{g \in G} \norm{\sigma_0 (g) - \sigma_1 (g)}
\]
so that $\sigma \mapsto A_{\sigma}$ is continuous. In particular, $A_{\sigma}$ is invertible if $\sigma$ is close to $\rho$. Now let $A_{\sigma} = B_{\sigma}\sqrt{A_{\sigma}^* A_{\sigma}}$ be the polar decomposition, with $B_{\sigma}$ unitary. The operator $\sqrt{A_{\sigma}^* A_{\sigma}}$ is equivariant as a map $H_{\rho} \to H_{\rho}$, and so $B_{\sigma}$ is equivariant as a map $H_{\rho} \to H_{\sigma}$, as claimed.

Now let $\tau \in X$ be arbitrary. Since $\pi$ is surjective, there is a unitary $u$ with $\tau = u \rho u^{-1}$. A section to $\pi$ near $\tau$ is given by $\sigma \mapsto u B_{u^{-1} \sigma u}$.
\end{proof}

\subsection{The Fredholm model for \texorpdfstring{$K$}--Theory, version 1}\label{fredholmsec1}

\begin{definition}
For each Hilbert space $H$, $\Fred(H)$ denotes the space of all bounded Fredholm operators $F: H \to H$, equipped with the operator norm topology.
Let $H$ be a $\Cl^{p,q}$-Hilbert space. A \emph{$\Cl^{p,q}$-Fredholm operator} on $H$ is a bounded, $\Cl^{p,q}$-linear, self-adjoint and odd Fredholm operator $F: H \to H$.
If $H$ is ample and $p-q \not \equiv -1 \pmod 4$, let $\Fred^{p,q}(H)$ be the space of all $\Cl^{p,q}$-Fredholm operators on $H$, equipped with the operator norm topology. If $p-q \equiv -1 \pmod 4$, then $\Fred^{p,q}(H)$ is the space of all such Fredholm operators $F$ with the property that the (self-adjoint) operator $\omega_{p,q} F \iota$ is neither essentially positive nor essentially negative.
The subspace of invertible elements is denoted $\cG^{p,q} (H) \subset \Fred^{p,q}(H)$. 
\end{definition}

The relevance of the condition in the case $p-q \equiv -1 \pmod 4$ is the following. One restricts $F$ to an invertible operator $F_0$ on $\ker (F)^{\bot}$. By a spectral deformation, $F_0$ is homotopic to an involution $f$. The datum $(\iota, e_i,f\iota , \eps_j)$ forms a $\Cl^{p+1,q}$-structure on $\ker (F)^{\bot}$, and the condition is that this structure is ample.
The following is a classical result due to Atiyah-Singer and Karoubi.

\begin{theorem}\label{atiyah-singer-karoubi1}\cite{AS69,Karoubi-paper}
If $H$ is an ample $\Cl^{p,q}$-Hilbert space, then $\Fred^{p,q} (H)$ is a representing space for the functor $\KO^{q-p}$.
\end{theorem}
Using the Morita equivalences, it is easy to derive Theorem \ref{atiyah-singer-karoubi1} from the version proven in \cite{AS69}.

\begin{lemma}\label{kuipertheorem}
If $H$ is ample, then $\cG^{p,q}(H)$ is contractible.
\end{lemma}
\begin{proof}
This follows without pain from Lemma \ref{kuipertheoremb}: inside $\cG^{p,q} (H)$, there is the subspace $\cG^{p,q}_{0}(H)$ of involutions. By a spectral deformation argument, the inclusion $\cG^{p,q}_{0}(H)\to \cG^{p,q}(H)$ is a homotopy equivalence (the details of this argument can be found in \cite{AS69}).
If $F \in \cG^{p,q}_{0}(H)$, then $F\iota$ defines an extension of the $\Cl^{p,q}$-structure to a $\Cl^{p+1,q}$-structure. Therefore, we can identify $\cG^{p,q}_{0} (H)$ with the fiber of the restriction map $\cS^{p+1,q} \to \cS^{p,q}$, and so it is contractible by Lemma \ref{kuipertheoremb}.
\end{proof}
By the symbol $\Omega \Fred^{p,q} (H)$, we denote the space of continuous paths $\gamma: \bI:=[-1,1] \to \Fred^{p,q}(H)$ such that $\gamma(\pm 1)\in \cG^{p,q}(H)$. This space indeed has the weak homotopy type of the homotopy-theoretic loop space of $\Fred^{p,q}(H)$, provided that $H$ is ample. The proof is a standard exercise in elementary homotopy theory, using Lemma \ref{kuipertheorem}.

For many practical purposes, the description of $KO$-Theory given in Theorem \ref{atiyah-singer-karoubi1} is not flexible enough. Before we can describe a more useful model, we need to say a few words about Fredholm families on Hilbert bundles.

\subsection{A digression on Hilbert bundles}\label{hilbertbundles}

For us, a \emph{Hilbert bundle} will always be a fiber bundle $V \to X$ with structural group $U(H)_{c.o.}$ and fiber $H$. 
Here, $H$ is a separable (possibly finite-dimensional) real Hilbert space and $U(H)_{c.o.}$ is its unitary group, with the compact-open topology. The fiber over $x \in X$ will be denoted $V_x$. The trivial Hilbert bundle with fiber $H$ over $X$ will be denoted by the symbol $H_X$. 
A $\Cl^{p,q}$-structure on $V$ will be described by a tuple of isometric automorphisms $\iota, e_1, \ldots, e_p, \eps_1, \ldots, \eps_q$ of $V$, satisfying the relations (\ref{clifford-relations}).
A Fredholm family $F$ on $V$ will be given by a collection $(F_x)_{x \in X}$, $F_x$ a Fredholm operator on $V_x$, and likewise a $\Cl^{p,q}$-Fredholm family will be a Fredholm family such that $F_x : V_x \to V_x$ is a $\Cl^{p,q}$-Fredholm operator for each $x \in X$. One needs a continuity condition on $x \mapsto F_x$, and the right formulation of it is a highly nontrivial insight by Dixmier and Douady \cite{DD}. It is given in Definitions \ref{defn-homomorphism}, \ref{defn-compactness} and \ref{definition:fredholmfamily} below (we were also strongly influenced by Atiyah-Segal \cite{Atiseg} and Kasparov's $\KK$-Theory).

\begin{definition}\label{defn-homomorphism}
A \emph{homomorphism} $F: V \to V'$ of Hilbert bundles is a fiber-preserving continuous map which is linear in each fiber.
\end{definition}
A homomorphism is determined by a collection $(F_x)_{x \in X}$ of bounded linear operators $F_x: V_x \to V'_x$. If $X$ is locally compact or metrizable, then the function $X \to \bR$, $x \mapsto \norm{F_x}$ is locally bounded, by the Banach-Steinhaus theorem (but in general not continuous).

\begin{example}\label{homomorphism-trivial-bundle}
Let $H$ be a separable Hilbert space, $X$ a compactly generated topological space and $F: X \to \Lin (H)$ be a map. The map $ H_X \to H_X$, $(x, v) \mapsto (x, F(x)v)$ is a homomorphism of Hilbert bundles if and only if $F$ is continuous when the target $\Lin (H)$ is equipped with the compact-open topology.
\end{example}

The collection $(F^*_x)_{x \in X}$ of adjoint operators may or may not form a homomorphism, and if it does, we say that $F$ is \emph{adjointable}. Note that a homomorphism which is pointwise self-adjoint is adjointable.
The set of adjointable endomorphisms of $V$ is a unital $*$-algebra $\Lin_X (V)$.
We say that $F:  V\to V$ is an isomorphism or invertible if it is a unit in $\Lin_X (V)$, in other words: $(F_x)_{x \in X}$ is an isomorphism if $F_x$ is an isomorphism for all $x \in X$ and if the collection $((F_x)^{-1})_{x \in X}$ is a homomorphism.
For self-adjoint homomorphisms $F: V \to V$, there is a functional calculus: if $F: V \to V$ is self-adjoint and $f: \bR \to \bR$ is a continuous function, then the collection $(f(F_x))_{x \in X}$ is a homomorphism, see \cite[p. 245]{DD}.

\begin{lemma}\label{invertibility-via-positivity}
Let $V \to X$ be a Hilbert bundle and let $F: V \to V$ be a self-adjoint homomorphism. Assume that there is a continuous function $\epsilon: X \to (0, \infty)$ such that $F_x^2 \geq \epsilon(x)^2$ for all $x \in X$. Then $F$ is invertible.
\end{lemma}
\begin{proof}
Invertibility of $F$ is a local (in $X$) property: if there is an open cover $\cU$ of $X$ such that $F|_U$ is invertible for all $U \in \cU$, then $F$ is invertible. Under the assumptions of the Lemma, there is an open cover $\cU$ of $X$ and for each $U \in  \cU$, there is $\epsilon>0$ such that $F_x^2 \geq \epsilon^2 >0$ for all $x \in U$. 
Let $f: \bR \to \bR$ be an odd continuous function with $f(t)=1/t$ for $|t| \geq \epsilon$. Then $f(F)$ is a homomorphism and for $x \in U$, we have $f(F_x)F_x=F_x f(F_x)=1$. Therefore, $F|_U$ is invertible for all $U \in \cU$ and hence $F$ is invertible.
\end{proof}

The notion of a \emph{compact operator} $V \to V$ is more difficult to formulate and we follow \cite[\S 22]{DD} (see also Proposition 3 loc.cit.). Let $\Gamma$ be the space of continuous sections of $V$, and for each $s, t \in \Gamma$, one defines the adjointable operator $\theta_{s,t} \in \Lin_X(V)$ by $\theta_{s,t} (v):= \langle s(x), v \rangle t(x)$ (for $v \in V_x$). The linear span of all these operators is a $*$-ideal $\Fin_X (V) \subset \Lin_X (V)$ (if $F$ is adjointable, then $F \theta_{s,t} = \theta_{s,Ft}$ and $\theta_{s,t} F = \theta_{F^* s ,t}$). Clearly, the operator $(\theta_{s,t})_x \in \Lin (V_x)$ is of finite rank and hence an element of $\Kom (V_x)$, the space of compact operators in $V_x$. 
\begin{definition}\label{defn-compactness}
An endomorphism $F=(F_x)_{x \in X}$ of $V$ is \emph{compact} if for each $x \in X$ and each $\epsilon>0$, there exists $G \in \Fin_X (V)$ and a neighborhood $U$ of $x$ such that for all $y \in U$, one has $\norm{F_y-G_y} \leq \epsilon$. The space of all compact operators is denoted $\Kom_X (V)$ and is a two-sided $*$-ideal in $\Lin_X (V)$.
\end{definition}
Note that if $F$ is compact, then $F_x$ is compact for all $x \in X$. It is not hard to characterize compact operators in trivial bundles.

\begin{lemma}
The compact operators of the trivial Hilbert bundle $H_X$ are precisely the continuous maps $X \to \Kom (H)$ (the target has the \emph{norm} topology, in contrast to the situation of Example \ref{homomorphism-trivial-bundle}).
\end{lemma}

\begin{definition}\label{definition:fredholmfamily}
Let $V \to X$ be a Hilbert bundle and $F\in \Lin_X(V)$. We say that $F$ is a \emph{Fredholm family} if there exists $G \in \Lin_X (V)$ such that $FG-1,GF-1 \in \Kom_X (V)$. We say that $G$ is a \emph{parametrix} to $F$.
If $V$ is equipped with a graded $\Cl^{p,q}$-structure, then a \emph{$\Cl^{p,q}$-Fredholm family} is a Fredholm family $F=(F_x)$ such that each $F_x$ is self-adjoint, Clifford linear and odd.
\end{definition}

We need criteria to check that a family is Fredholm. The first useful thing to know is that on reasonable spaces, being Fredholm is a local property (it is not a pointwise property).

\begin{lemma}\label{gluing-lemma-fredholm}
Let $V \to X$ be a Hilbert bundle over a paracompact space and $F \in \Lin_X (V)$. Assume that there is an open cover $\cU$ of $X$ such that $F|_U$ is Fredholm for all $U \in \cU$. Then $F$ is Fredholm.
\end{lemma}
\begin{proof}
Let $G_U$ be a parametrix for $F|_U$. Let $(\lambda_U)_{U \in \cU}$ be a partition of unity subordinate to $\cU$. Then $\sum_{U \in \cU }\lambda_U  G|_U$ is a parametrix to $F$ (since the property of being compact is local in $X$).
\end{proof}

\begin{lemma}\label{local-to-global-fredholm-families}
Let $V \to X$ be a Hilbert bundle over a space. Let $(F_x)_{x \in X}$ be a collection of bounded operators $F_x \in \Lin (V_x)$. Assume that each $x \in X$ admits a neighborhood $U$ of $x$ and a trivialization $V|_U\cong U \times H$ such that in this trivialization $F$ is given by a continuous map $X \to \Lin (H)$ (the target has the norm topology). Then $F \in \Lin_X(V)$. If $F_x$ is compact (invertible) for each $x \in X$, then $F$ is compact (invertible). If $F_x$ is Fredholm for all $x \in X$, then $F$ is Fredholm, provided that $X$ is paracompact.
\end{lemma}
\begin{proof}
The statements on adjointability, compactness and invertibility are trivial consequences of what we have said so far. Next, recall that the quotient map $\Lin (H) \to \Lin (H)/\Kom (H)$ has a continuous section (both spaces are equipped with the norm topology). This follows from a general theorem by Bartle and Graves which states that if $\phi:X \to Y$ is a surjective bounded operator between Banach spaces, then there is a continuous cross-section $\sigma:Y \to X$ (\cite{BartGrav}, see also \cite[p. 187]{Holm} for the explicit statement). Using this section, one constructs a map $\sigma:\Fred (H) \to \Fred(H)$ such that $\sigma (F)$ is a parametrix for $F$, for each $F \in \Fred(H)$. 
Therefore, if $F_x$ is Fredholm for all $x \in X$, then $F$ is locally Fredholm. By Lemma \ref{gluing-lemma-fredholm}, it follows that $F$ is globally Fredholm.
\end{proof}

\begin{lemma}\label{fredholm-via-positivity}
Let $V \to X$ be a Hilbert bundle, $X$ paracompact and let $F \in \Lin_X (V)$ be self-adjoint. 
The following statements are equivalent.

\begin{enumerate}
\item $F$ is a Fredholm family.
\item There is a continuous function $\epsilon: X \to (0, \infty)$ such that $F^{2} \geq \epsilon^2 \bmod \Kom_X (V)$ in the sense that there exists $K \in \Kom_X (V)$ with $F^2 + K \geq \epsilon^2$ (pointwise).
\end{enumerate}
\end{lemma}
\begin{proof}
If $F$ is Fredholm, then there is a parametrix $G$, and we can assume that $G$ is self-adjoint. Then $K= FG-1$ and $K^* = GF-1$ are compact. For $x \in X$, we pick $\delta>0$ and a neighborhood $U$ of $x$ such that $\norm{G_y} \leq \delta$ and hence $G_y^2 \leq \delta^2$ for $y \in U$ (the norm of a homomorphism is locally bounded). It follows that $F_y \delta^2 F_y \geq F_y G^2_y F_y = (1+K_y)(1+K_y^*)=:1+L_y$ and therefore $F_y^2 \geq \delta^{-2} (1+L_y)$ throughout $U$. Put $\epsilon_x := \delta^{-1}$. This proves the desired inequality locally. Globally, one patches the constant functions $\epsilon_x$ on these neighborhoods together with a partition of unity.

For the reverse implication, pick $K$ and $\epsilon$ as in (2). Then $F^2 + K$ is invertible (Lemma \ref{invertibility-via-positivity}) and $F F (F^2 +K)^{-1} -1=- K (F^2+K)^{-1}$ is compact, proving that $F (F^2 +K)^{-1}$ is a right-parametrix of $F$. Similarly, $(F^2+K)^{-1}F$ is a left-parametrix.
\end{proof}

\subsection{The Fredholm model for \texorpdfstring{$K$}--Theory, version 2}\label{fredholmsec2}

We denote the \emph{product of space pairs} by $(X,Y) \times (Z,W):= (X \times Z, X \times W \cup Y \times Z)$.

\begin{definition}\label{defn-fpq}
Let $(X,Y)$ be a space pair. A \emph{$(p,q)$-cycle} on $(X,Y)$ is a pair $(V,F)$, where $V$ is a $\Cl^{p,q}$-Hilbert bundle (with separable fibers) and $F$ a $\Cl^{p,q}$-Fredholm family on $V$. Moreover, we require that $F|_Y$ is invertible.
There are obvious notions of \emph{pullbacks}, \emph{direct sum} and \emph{isometric isomorphisms} of $(p,q)$-cycles. 
Two $(p,q)$-cycles $(V_0,F_0)$ and $(V_1,F_1)$ are \emph{homotopic} or \emph{concordant} if there exists a $(p,q)$-cycle $(V,F)$ on $(X\times [0,1],Y\times [0,1])$ such that the restriction of $(V,F)$ to $X \times \{i\}$ is isomorphic to $(V_i,F_i)$. Homotopy is an equivalence relation.
A $(p,q)$-cycle is \emph{acyclic} if $F$ is invertible. The set of homotopy classes of $(p,q)$-cycles is an abelian monoid, and we define the abelian group $F^{p,q}(X,Y)$ as the quotient of that monoid by the submonoid of homotopy classes that contain acyclic representatives.
\end{definition}

We implicitly said in the definition that $F^{p,q}(X,Y)$ is a group, and this is indeed true. For this and similar purposes, we use the following convenient criterion to prove that two $(p,q)$-cycles are concordant. We say that a space pair $(X,Y)$ is a \emph{paracompact pair} if $X$ is paracompact and $Y \subset X$ is closed. We denote the \emph{anticommutator} of two operators by
\[
\{F_0, F_1\}:= F_0 F_1 + F_1 F_0.
\]

\begin{lemma}\label{nullhomotopy-criterion}
Let $(V,c,\iota,F_0)$ and $(V,c,\iota,F_1)$ be two $(p,q)$-cycles on the paracompact space pair $(X,Y)$. If the anticommutator $\{F_0, F_1\}$ is nonnegative, i.e. $\{(F_0)_x ,(F_1)_x\} \geq 0$ for all $x \in X$, then $[V,c,\iota,F_0]=[V,c,\iota,F_1] \in F^{p,q}(X,Y)$.
\end{lemma}

This is a generalization of a special case of a result in $\KK$-theory due to Connes and Skandalis \cite[Proposition 17.2.7]{Bla}.
One important feature of Lemma \ref{nullhomotopy-criterion} is that the positivity condition can be checked pointwise, hence often the proof for general $X$ is only notationally more involved than that for $X=*$.

\begin{proof}
The homotopy is given by $P_{s}:=\cos (s) F_0 + \sin (s) F_1$, $s \in [0, \pi/2]$. Compute
\[
 P_{s}^{2}= \cos(s)^2 F_0^2 + \sin(s)^2 F_1^2 + 2 \sin(s) \cos(s) (F_0 F_1+F_1 F_0)  \geq  \cos(s)^2 F_0^2 + \sin(s)^2 F_1^2 .
\]
There is a function $\epsilon: X \to (0,\infty)$ such that $F_i^2 \geq \epsilon^2 \bmod \Kom_X (V)$ (Lemma \ref{fredholm-via-positivity}) and such that $(F_i^2)|_Y \geq \epsilon^2|_Y$. Therefore $P_{s}^{2} \geq \epsilon^2 \bmod \Kom_X (V)$ and $P_s$ is a Fredholm family on $X \times [0,\pi/2]$ by Lemma \ref{fredholm-via-positivity}. Moreover, the restriction of $P$ to $Y \times [0,\pi/2]$ is invertible by Lemma \ref{invertibility-via-positivity}.
\end{proof}

\begin{lemma}\label{additiveinverse}
Let $(X,Y)$ be a paracompact pair. Then $F^{p,q}(X,Y)$ is a group, and the additive inverse to $[V,c,\iota,F]$ is $[V,-c,-\iota,F]$ or $[V,c,-\iota,-F]$.
\end{lemma}
\begin{proof}
We claim that $[V \oplus V,c \oplus -c,\iota\oplus -\iota,F\oplus F]=0$. The operator $Q:=\twomatrix{}{\iota}{\iota}{}$ is an odd, $\Cl^{p,q}$-linear involution that anticommutes with $F \oplus F$ and a straightforward application of Lemma \ref{nullhomotopy-criterion} concludes the proof. The second formula is proven in the same way with $Q= \twomatrix{}{1}{1}{}$.
\end{proof}

How do these new groups relate to $\KO$-Theory? 

\begin{theorem}\label{compact-comparison-theorem}
Let $(X,Y)$ be a compact pair. Then the obvious comparison map
\[
[(X,Y);(\Fred^{p,q} (H),\cG^{p,q} (H))] \to F^{p,q} (X,Y).
\]
is bijective.
\end{theorem}

Moreover, the functor $F^{p,q}$ is representable for paracompact pairs.

\begin{theorem}\label{paracompact-comparison-theorem}
There exists a space pair $(K^{p,q},D^{p,q})$ and a natural map 
\[
[(X,Y); (K^{p,q},D^{p,q})] \to F^{p,q} (X,Y)
\]
which is bijective whenever $X$ is paracompact and compactly generated and $Y \subset X$ closed. Moreover, the space $D^{p,q}$ is weakly contractible.
\end{theorem}

Theorems \ref{compact-comparison-theorem} and \ref{paracompact-comparison-theorem} are proven in the appendix \ref{appendix}.
From both theorems, one obtains a comparison map 
\[
(\Fred^{p,q} (H),\cG^{p,q} (H)) \to (K^{p,q},D^{p,q})
\]
and a formal consequence of the above results is that it is a weak homotopy equivalence (because $\Fred^{p,q} (H)$ and $\cG^{p,q} (H)$ are by definition metric spaces and therefore paracompact).
Let us have a look at Bott periodicity in this framework. We write $\bI:=[-1,1]$, $\partial \bI:=\{\pm 1 \}$ and
\[
\Omega F^{p,q} (X,Y) := F^{p,q} ((\bI, \partial \bI) \times (X,Y)).
\]
Let $p \geq 1$ and $(V,F)$ be a $(p,q)$-cycle over $(X,Y)$. Put $J:=( e_1 \iota)$. Define a $(p-1,q)$-cycle $\Bott(V,F)$ on $(\bI, \partial \bI) \times (X,Y)$ by the formula
\[
\Bott(F)_{(s,x)} := F_x + s J.
\]

The operator $J$ is self-adjoint and odd and hence so is $\Bott(F)_{(s,x)}$. Since $J$ anticommutes with each $F_x$, the formula $(\Bott(F)_{(s,x)})^2 = F_x^{2}+ s^2$ holds and $\Bott(F)_{(s,x)}$ is invertible if $s \neq 0$ or if $F$ is invertible. Observe that $\Bott(F)_{(s,x)}$ commutes with $\eps_1, \ldots, \eps_q$ and $e_2, \ldots,e_p$, but \emph{not} with $e_1$. When we define a $\Cl^{p-1,q}$-action by reindexing the $e_i$ generators, $\Bott (F)_{(s,x)}$ is $\Cl^{p-1,q}$-linear. This construction preserves the equivalence relation defining the $F$-groups, and we obtain a group homomorphism, the \emph{Bott map}
\begin{equation}\label{bott-map2}
\bott: F^{p,q}(X,Y) \to\Omega  F^{p-1,q}(X,Y).
\end{equation}

The Bott periodicity theorem states that (\ref{bott-map2}) is an isomorphism for all paracompact and compactly generated pairs; this follows from the main result of \cite{AS69}, together with Theorems \ref{compact-comparison-theorem} and \ref{paracompact-comparison-theorem}. Using the Morita equivalences, we obtain natural isomorphisms of functors
\[
\mor^{1,1}:F^{p,q} \cong F^{p+1,q+1}, \;\mor^{8,0}: F^{p,q} \cong F^{p+8,q}, \; \mor^{0,8}:F^{p,q} \cong F^{p,q+8}.
\]

The formulas for the Morita equivalences that we gave make it clear that these maps are compatible with the Bott maps whenever this statement makes sense, i.e. if $p>0$.
Another fundamental structure is a $\KO^{0}$-module structure on $F^{p,q}(X,Y)$. To define it, we use the model for $\KO^0 (X,Y)$ by complexes of vector bundles, see \cite[Part II]{ABS}. We can slightly reformulate this construction by saying:

\begin{definition-proposition}
Let $(X,Y)$ be a compact pair. Consider the monoid of concordance classes of $(0,0)$-cycles such that the underlying Hilbert bundle has finite rank and the same equivalence relations as in Definition \ref{defn-fpq}. The quotient group is naturally isomorphic to $\KO^0 (X,Y)$.
\end{definition-proposition}

The natural map $\KO^0 (X,Y) \to F^{0,0}(X,Y)$ is an isomorphism for compact pairs $(X,Y)$. For $[H,\iota,c,F] \in F^{p,q} (Z,W)$ and $[E,\eta,P] \in \KO^0 (X,Y)$, we define
\[
[H,\iota,c,F] \cdot [E,\eta,P]:= [H \otimes E, \iota \otimes \eta, c \otimes 1, F \otimes 1 + \iota \otimes P].
\]
It is easy to see that this definition yields a well-defined bilinear map
\[
F^{p,q}(Z,W) \times \KO^0  (X,Y) \to F^{p,q} ((Z,W)\times (X,Y)).
\]
If $(p,q)=(0,0)$, this defines the usual cross product on $\KO^0$ under the isomorphism from Theorems \ref{compact-comparison-theorem} and \ref{paracompact-comparison-theorem}. This product is compatible with Morita equivalences and Bott periodicity, in the sense that these maps are linear over $\KO^0$.

\section{Analytical arguments}\label{analysis}

\subsection{Preliminaries on elliptic theory}\label{elliptic-preliminaries}

In this subsection, we recollect some purely analytical results from the literature. It was written for the convenience of the readers with less background in analysis (including the author), and we also want to give a reference for further analytical results needed in \cite{BERW}, which is why we formulate the analysis in greater generality. 

Even though our main result only involve differential operators, the proof requires to use the bigger class of pseudo-differential operators. Let us first summarize the essential properties of pseudo-differential operators that we will use, following Atiyah and Singer \cite{AS1,ASIV}. Let $M$ be a closed manifold and let $E \to M$ be a smooth vector bundle.
In \cite[\S 5]{AS1}, a vector space $\pseudo{m} (E)$ of linear maps $\Gamma (M;E)\to \Gamma (M;E)$, called pseudo-differential operators of order $m \in \bZ$ (denoted $\mathcal{P}^m(M;E,E)$ in loc. cit.), is considered whose properties we briefly recall.
\begin{proposition}\label{pseudo-properties}\mbox{}
\begin{enumerate}
\item $\pseudo{m}(E) \subset \pseudo{m+1}(E)$ and $\bigcup_{m \in \bZ} \pseudo{m} (E)$ is a filtered algebra. Pseudo-differential operators have (formal) adjoints.
\item Pseudo-differential operators have symbols. Let $\pi: T^* M \setminus 0 \to M$ be the cotangent bundle without the zero section. The leading symbol of $P \in \pseudo{m}(E)$ is a smooth section $\symb_P $ of $\pi^* \Hom (E,E) \to T^* M \setminus 0$ which is \emph{positively homogeneous of degree $m$} in the sense that $\symb_P (t \xi)= t^m \symb_P (\xi)$, for all $0 \neq \xi \in T^* M$ and $t>0$. There is an associated notion of ellipticity. 
Moreover, $\symb_P  (\xi) \circ \symb_Q (\xi)= \symb_{P \circ Q} (\xi)$ and $\symb_{P^*} (\xi) = \symb_P (\xi)^*$. 
\item A differential operator of order $m$ is in $\pseudo{m} (E)$.
\item The inverse of an invertible elliptic $P \in \pseudo{m} (E)$ is in $\pseudo{-m}(E)$.
\item Convolution operators with smooth kernels lie in $\pseudo{-\infty} (E) = \bigcap_m \pseudo{m} (E)$.
\item Let $P \in \pseudo{m} (E)$. For each $s \in \bR$ and $m \leq r \in \bR$, $P$ induces a bounded operator of the $L^2$-Sobolev spaces $W^s (M;E) \to W^{s-r} (M;E)$, whose norm we denote by $\norm{P}_{s,s-r}$. 
\item There is a Fr\'echet space structure on $\pseudo{m} (E)$, defined in \cite[p. 123 f.]{ASIV}. The map $\pseudo{m} (E) \to \Lin (W^s (M;E) , W^{s-r} (M;E))$ is continuous for all $s$, $r \geq m$. 
\item Let $\Diff (M;E)$ be the group of pairs $(f, \hat{f})$, $f$ a diffeomorphism of $M$, and $\hat{f}$ is a smooth bundle isomorphism $E \to E$ covering $f$. This is a topological group with the topology described in \cite[p. 123]{ASIV}. The group $\Diff (M;E)$ acts linearly on $\pseudo{m}(E)$ by pulling back, and this action is continuous (i.e. $\Diff (M;E) \times \pseudo{m}(E) \to \pseudo{m}(E) $ is continuous), compare \cite[Proposition 1.3]{ASIV}.
\end{enumerate}
\end{proposition}
We have to consider pseudo-differential operators on noncompact manifolds as well, but only in a very restricted situation. Let $E \to M$ be a smooth vector bundle on a smooth manifold and let $U \subset M$ be an open and relatively compact subset. We only consider pseudo-differential operators $P:\Gamma_c (M;E) \to \Gamma_c (M;E)$ such that
\begin{enumerate}
\item $\supp (Pu) \subset U$ for all $u \in \Gamma_c (M;E)$ and
\item there is a function $\lambda \in C^{\infty}_{c} (U)$ such that $P((1-\lambda)u)=0$ for all $u \in \Gamma_c (M;E)$.
\end{enumerate}
The space of these operators is denoted $\pseudo{m}(E)_U$. It has all the properties listed above; the group $\Diff (M;E)$ needs to be replaced by the group $\Diff (M;E)_U$ of pairs $(f, \hat{f})$ such that $f$ is the identity outside $U$. A simple method to obtain the Fr\'echet space structure is to take a compact submanifold $U \subset K \subset M$ of codimension $0$ and to take the double $K'$ of $K$ with $E' \to K'$ the double of $E$. Extension by zero yields a linear embedding $\pseudo{m}(E)_U \to \pseudo{m}(E')$. We will consider the following specific situation.

\begin{assumptions}\label{analysis-input}\mbox{}
\begin{enumerate}
\item Let $N$ be a $d$-dimensional manifold, equipped with a proper smooth map $t: N \to \bR$ and let $E \to N$ be a smooth vector bundle. For a subset $A \subset \bR$, we denote $N_A:= t^{-1} (A)\subset N$ and $E_A= E|_{N_A}$. We assume the following conditions on $N$, $t$ and $E$. 
\item There are $(d-1)$-dimensional closed manifolds $M_{\pm}$, smooth vector bundles $V_{\pm} \to M_{\pm}$ and numbers $r_-<R_- < R_+ < r_+ \in \bR$ such that $N_-:=N_{(-\infty,R_-)}= (-\infty,R_-)\times M_-  $, $E_-:=E_{(-\infty,R_-)}=  (-\infty,R_-)\times V_-$ (as spaces over $(-\infty,R_-)$) and the analogous condition for $N_+:=N_{(R_+,\infty)}$ hold. We write $N_0 := N_{(r_-,r_+)}$ and $E_0:= E|_{N_0}$. The data so far introduced is called a \emph{manifold with cylindrical ends}, together with a vector bundle with cylindrical ends.
  \item A Riemannian metric $g$ on $N$ is \emph{cylindrical at the ends} if there are Riemannian metrics $g_{\pm}$ on $M_{\pm}$ such that $g = dt^2+g_{\pm} $ on $N_{\pm}$. Likewise, a bundle metric $h$ on $E$ is cylindrical at the ends if there are metrics $h_{\pm}$ on $V_{\pm}$ such that the equality $E_- = (-\infty,R_-) \times V_-$ and the analogous equality over $(R_+,+\infty)$ are equalites of Riemannian vector bundles.
\item Next, let $\sigma = (\sigma_-, \sigma_+)$, with $\sigma_{\pm}$ being a skew-adjoint endomorphism of $V_{\pm}$ such that $\sigma_{\pm}^{2} = -1$. 
\item The last piece of datum is a linear map $D: \Gamma_c (N;E) \to \Gamma_c (N;E)$ which is symmetric and has the following properties.
\item $D$ maps each of the spaces $\Gamma_c (N_0;E)$ and $\Gamma_c (N_{\pm}; E)$ to itself.
\item The restriction of $D$ to $\Gamma_c (N_0;E)$ is a formally self-adjoint first order elliptic pseudo-differential operator of order $1$.
\item There are \emph{invertible} formally self-adjoint first order elliptic pseudo-differential operators $B_{\pm}$ on $V_{\pm} \to M_{\pm}$ such that $\sigma_{\pm} B_{\pm} + B_{\pm} \sigma_{\pm}=0$ and such that the operators $D_{\pm} =\partial_t \sigma_{\pm}+ B_{\pm}$ on $\Gamma_c (N_{\pm}; E)$ coincide with the restriction of $D$ to $\Gamma_c (N_{\pm}; E)$. To understand this formula, note that an element in $\Gamma_c (N_{\pm};E)$ can be viewed as a smooth function $\bR \to \Gamma (M_{\pm};V_{\pm})$.
\end{enumerate}
\end{assumptions}

Note that at least in a formal sense, one can view $D$ as a pseudo-differential operator. Also, the operators $B_{\pm}$ are determined by $D$.
We denote $\cH := L^2(N;E)$, $\cW:= W^{1} (N;E)$, $H_{\pm}= L^2 (M_{\pm};V_{\pm})$ and $W_{\pm}=W^{1}(M_{\pm};V_{\pm})$.
The norm on $L^2 (N;E)=W^{0}(N;E)$ is given by the Riemannian metric on $N$ and the bundle metric on $E$. 
To define the Sobolev $1$-norm on compactly supported smooth sections $\Gamma_c (N;E)$, recall first for each relatively compact subset $U \subset N$, the Sobolev norm on $\Gamma_c (U;E)$ is unique up to equivalence of norms.
Let $\mu_+, \mu_-, \mu_0$ be smooth nonnegative functions on $\bR$ such that $\mu_- + \mu_+ + \mu_0=1$, such that $\mu_0$ is supported in $(r_-,r_+)$, $\mu_-$ is supported in $(-\infty,R_-)$ and $\mu_+$ is supported in $(R_+, \infty)$. 
To avoid heavy notation, we use the symbol $\mu_i$ for the function $\mu_i \circ t$ on $N$, which is supported in $N_i$ (and $\mu_0$ is compactly supported). The Sobolev $1$-norm $\norm{u}_{\cW}$ of $ u \in \Gamma_c (N; E)$ is defined by
\[
\| u \|_{\cW}^{2} := \| \mu_0 u\|_{W^1(N_0)}^{2} +  \int_{-\infty}^{R_-} (\| u(t)\|_{W_-}^{2} + \| \dot u (t)\|_{H_-}^{2} )dt + \int^{\infty}_{R_+}  (\| u(t)\|_{W_+}^{2} + \| \dot u (t)\|_{H_+}^{2} )dt
\]
and its equivalence class does not depend on the specific choice of the functions $\mu_i$.
\begin{proposition}\label{fredholmproperty:pointwise}\mbox{}
\begin{enumerate}
\item There are constants $C_1, C_2, C_3$, such that for all $u \in \cW$, the elliptic estimate holds:
\begin{equation}\label{elliptic-estimate}
\|u \|_{0} \leq C_1 \| u \|_{1} \leq C_2 (\|\mu_0 u\|_{0} + \| Du\|_{0}) \leq C_3 (\| u\|_{0} + \| Du\|_{0}).
\end{equation}
\item The operator $D$ with $\dom (D)=\cW$ is a self-adjoint (unbounded) Fredholm operator. 
\end{enumerate}
\end{proposition}
\begin{proof}
The first and third inequality in (\ref{elliptic-estimate}) are clear. For the second inequality, let $u $ be supported in $N_+$. Then
\[
\| Du \|^{2}_{\cH} = \int_{R_+}^{\infty} \| B_+ u (t) \|_{H_+}^{2} + \| \dot u(t)\|_{H_+}^{2} dt
\]
since $B_+$ and $\sigma_+$ anticommute. Since $B_+$ is invertible, there is a constant $c >0$ such that $\| B_+ u (t) \|_{H_+}^{2} \geq c \| u (t) \|_{W_+}^{2}$. Altogether, we obtain $\|u\|_{\cW} \leq C \| Du\|_{\cH}$ for such sections. A similar estimate holds for sections supported in $N_-$. For sections supported in $N_0$, the usual elliptic estimate (G{\aa}rdings inequality) holds. One patches all these estimates together by an argument similar to that for \cite[Lemma 3.7]{RoSa}.

The proof for essential self-adjointness is the same as that for differential operators given in \cite[Lemma 10.2.1, 10.2.5, Proposition 10.2.10]{HR}.
The elliptic estimate, Rellich's theorem and a general fact \cite[Lemma 3.5]{RoSa} imply that $D: \cW \to \cH$ has closed image and finite-dimensional kernel. Since $D$ is self-adjoint, the orthogonal complement of the image of $D$ is the same as the kernel, and so $D$ is Fredholm.
\end{proof}

\subsection{Families of pseudo-differential operators with cylindrical ends}

We will, very crucially for our purposes, study families of pseudo-differential operators on bundles of manifolds. Let us first state precisely what we mean by this term. Let $N,t,M_{\pm},E,V_{\pm}$ be as in Assumption \ref{analysis-input}, but without metrics and choose functions $\mu_i$ as above. 
Let $\fF$ be the set of all tuples $(g,h,\sigma,D)$ as in Assumption \ref{analysis-input}. The assignment $(g,h,\sigma,D) \mapsto (g,h,\sigma_-, \sigma_+, \mu_0 D_0 \mu_0, B_-,B_+)$ is an injective map from $\fF$ to the Fr\'echet space
\[
\begin{split}
\Gamma (N; \Sym^2 T^*N) \times \Gamma (N; \Sym^2 E^*) \times \Gamma (M_-; \End (V_-) )\times \Gamma (M_+; \End (V_+)) \times \\
\times \pseudo{1} (N;E)_{N_0} \times \pseudo{1} (M_-;V_-) \times  \pseudo{1} (M_+;V_+).
\end{split}
\]
We equip $\fF$ with the subspace topology induced by this map.
Let $G$ be the topological group of all pairs $(f,\hat{f})$, where $f: N \to N$ is a diffeomorphism which is the identity on $N_-\cup N_+$ and $\hat{f}: E \to E$ is a bundle automorphism covering $f$ such that on the parts $E_{\pm}$, $\hat{f}$ is constant in the $\bR$-direction. This group has the $C	^{\infty}$-topology, and it acts on the space $\fF$ by pullbacks. This action is continuous, by \cite[Proposition 1.3]{ASIV}.

Let $X$ be a paracompact space and $Q \to X$ be a $G$-principal bundle. Let $Z := Q \times_G N \to X$; this is a bundle of manifolds ``with cylindrical ends'' (note that the ``end bundle'' is trivial). On the total space $Z$, there is an associated vector bundle $Q \times_G E \to Z$. We denote the fiber of $Z \to X$ over $x$ by $Z_x$ and the restriction of $Q \times_G E$ to $Z_x$ by $E_x$. 

\begin{definition}\label{pseudodifferential-family}
Let $Z \to X$ be as above. A \emph{family of formally self-adjoint order one elliptic pseudo-differential operators with cylindrical ends and invertible at infinity} is a section of the fiber bundle $Q \times_G \fF \to X$. 
\end{definition}

Note that such a family induces Riemannian metrics on the fiber $Z_x$ and bundle metrics on the vector bundles $E_x \to Z_x$. 
Therefore, we can form the Hilbert spaces $L^2 (Z_x; E_x)$ of $L^2$-sections. We now show how to arrange all these Hilbert spaces in a Hilbert bundle. To this end, we first consider the case of topologically trivial bundles.

Let $\fM$ be the space of all $(g,h)$ as in Assumption \ref{analysis-input}, let $(g_0,h_0) \in \fM$ be a basepoint and let $L^2 (N;E)_{g,h}$ be the Hilbert space of sections with respect to the inner product induced by the Riemannian metric $g$ and the bundle metric $h$. We want to define a Hilbert bundle structure on the disjoint union
\[
\coprod_{(g,h)\in \fM} L^2 (N;E)_{g,h} \to \fM
\]
(which does not yet have a topology). This is based on the following observation.

\begin{observation}\label{gauging-lemma}
Let $V$ be a finite-dimensional real vector space and $h_0$ be an inner product on $V$. For each other inner product $h$ on $V$, there exists a unique $h_0$-self-adjoint and positive definite endomorphism $\tau_{h_0,h}$ of $V$ such that 
\[
h (v,w)=h_0 (\tau_{h_0,h} v, \tau_{h_0,h} w)
\]
for all $v, w \in V$. Since $\tau_{h_0,h}$ depends smoothly on $h_0$ and $h$, this carries over to the bundle situation.

Similarly, if $N$ is a manifold and $g_0$ a Riemannian metric on $N$, then for any other Riemannian metric $g$ on $N$, let $\alpha_{g_0,g}: N \to (0, \infty)$ be the function with $\alpha_{g_0,g}^{2} \vol_{g_0} = \vol_g$, where $\vol_g$ denotes the volume measure induced by $g$. The function $\alpha_{g_0,g}$ depends smoothly on the choice of $g_0$ and $g$. 

Let $E \to N$ be a real vector bundle. Let $L^2 (N;E)_{g,h}$ be the Hilbert space of sections with respect to the inner product induced by the Riemannian metric $g$ and the bundle metric $h$. 
Then for each $s,t \in \Gamma_c (N;E)$, we have 
\[
\langle s,t \rangle_{g,h} = \int_N h( s, t ) d \vol_g = \int_N h_0 (\tau_{h_0,h} s, \tau_{h_0,h} t) \alpha_{g_0,g}^{2} d \vol_{g_0} = \langle  \alpha_{g_0,g}\tau_{h_0,h} s,   \alpha_{g_0,g} \tau_{h_0,h} t\rangle_{g_0,h_0} .
\]
In other words, $s \mapsto \alpha_{g_0,g} \tau_{h_0,h} s$ defines an isometry 
\[
T_{(g,h),(g_0,h_0)}:L^2 (N;E)_{g,h} \to L^2 (N;E)_{g_0,h_0}. 
\]
\end{observation}

Now let
\[
T= \coprod_{(g,h)\in \fM} T_{(g,h),(g_0,h_0)}: \fL=\coprod_{(g,h)\in \fM} L^2 (N;E)_{g,h} \to \fM \times L^2 (N;E)_{g_0,h_0};
\]
this is bijective and fiberwise an isometry. We \emph{declare} $T$ to be a homeomorphism; this defines a bundle structure on $\fL \to \fM$.
The group $G$ acts on $\fM$ by pullbacks; the precise formula is $(f,\hat{f}) \cdot (g,h) := ((f^{-1})^* g, (\hat{f}^{-1})^* h)$. This action is covered by an isometric action of $G$ on the Hilbert bundle $\fL$: define 
\[
\widetilde{(f, \hat{f})} : L^2 (N;E)_{(g,h)} \to L^2 (N;E)_{((f^{-1})^* g, (\hat{f}^{-1})^* h)}
\]
by 
\[
s \mapsto \hat{f} (s(f^{-1}(x))).
\]
It is straightforward to check that this defines an action of $G$ on $\fL$, and that $\widetilde{(f, \hat{f})}$ is an isometry.
We obtain a Hilbert bundle 
\[
Q \times_G \fL \to Q \times_G \fM,
\]
and given a section of $Q \times_G \fM \to X$ (i.e. a choice of metrics on the fibers $Z_x$ and $E_x$), we can pull back this bundle to a Hilbert bundle on $X$. To avoid heavy notation, we write $\cH \to X$ for this Hilbert bundle. Inside $\cH_x$, there is the Sobolev space, denoted $\cW_x$. 

Given a family of pseudo-differential operators as in Definition \ref{pseudodifferential-family}, we can form the individual self-adjoint unbounded Fredholm operators $D_x: \cW_x \to \cH_x$ (using Proposition \ref{fredholmproperty:pointwise}). Using the functional calculus, we may form the bounded transforms $\normalize{D_x}$, which are bounded and self-adjoint Fredholm operators. Altogether, we obtain a collection $x \mapsto \normalize{D_x}$ of self-adjoint Fredholm operators on the fibers of $\cH \to X$. We wish to show that this is a Fredholm family.

\begin{proposition}\label{fredholmproperty:familywise}
The family $x \mapsto \normalize{D_x}$ is a self-adjoint Fredholm family (in the sense of Definition \ref{definition:fredholmfamily}) on the Hilbert bundle $\cH \to X$. Moreover, the family has the stronger continuity condition described in Lemma \ref{local-to-global-fredholm-families}.
\end{proposition}

The proof is carried out in three steps. The overall structure is best understood in a formal context. There is a tautological bundle $\fF \times N \to \fF$, and this carries a tautological family of pseudo-differential operators. For a given Riemannian metric $g$ on $N$ and bundle metric $h$ on $E$ (always satisfying the conditions from Assumption \ref{analysis-input}), we denote by $\fF_{g,h} \subset \fF$ the subspace of tuples $(g,h,\sigma,D)$. 

The first step of the proof is to prove the result for the tautological family over $\fF_{g,h}$. In that case, the Hilbert bundles $\cH$ is trivialized, and the domains of all the unbounded operators are equal.
The second step is to prove the result for the tautological family over $\fF$, which is reduced to the first step by a ``gauging trick''. The third step is to use general nonsense to reduce the case of arbitrary bundles to the tautological family.
For the first step, we use the following criterion by Nicolaescu for the continuity of the bounded transform of a family of unbounded operators.
\begin{lemma}\label{nicolaescu-lemma}\cite[Proposition 1.7]{Nic}
Assume that $\cH$ is a Hilbert space, $D_n$, $n \in \bN$ and $D$ are self-adjoint unbounded operators on $\cH$. Assume that
\begin{enumerate}
\item $D$ has a spectral gap.
\item The domains of all the operators $D_n - D$ contain the domain of $D$.
\item There is a sequence of positive numbers $c_n \to 0$ with $\|( D-D_n) u\| \leq c_n (\| u\| + \| D u\|)$ for all $u \in \dom (D)$.
\end{enumerate}
Then $D_n \to D$ is ``Riesz convergent'', which by definition \cite[\S 1]{Nic} means that $\normalize{D_n} \to \normalize{D}$ is norm convergent. 
\end{lemma}
\begin{proof}[Proof of Proposition \ref{fredholmproperty:familywise}, first step]
We first prove that the map $X=\fF_{g,h} \to \Fred (L^2 (M;E))$, $x \mapsto \normalize{D_x}$ is continuous (the target has the norm topology induced by $g$ and $h$).
Because $X$ is a metric space, it suffices to prove sequential continuity. 
Let $D_n\to D $ be a sequence of operators associated with the convergent sequence $x_n \to x$ in $X$.
The operators $D_n$ and $D$ all have the same domain $\cW$. Thus the second condition of Lemma \ref{nicolaescu-lemma} holds. The first condition of Lemma \ref{nicolaescu-lemma} holds in our situation. If $0$ is not in the spectrum of $D$, there is nothing to show, and if $0$ is in the spectrum of $D$, then it is an isolated eigenvalue because $D$ is a self-adjoint Fredholm operator. In that case, any sufficiently small $\lambda>0$ does not lie in the spectrum of $D$, hence $D$ has a spectral gap.
For the third condition, observe that the sequences $\norm{\mu_0 (D_n-D)}_{1,0}$ and $\norm{B_{\pm,n}-B_{\pm}}_{1,0}$ are null sequences (use Proposition \ref{pseudo-properties}). Likewise, the operator norm of $\sigma_{\pm,n}-\sigma_{\pm}$ on $H_{\pm}$ converges to zero. Estimate
\begin{equation}\label{estimate-for-nicolaescu-lemma}
\norm{ (D_n - D)u}_{\cH} \leq \norm{\mu_0 (D_n - D)u}_{\cH} + \norm{\mu_{+} (D_n - D)u}_{\cH}  + \norm{\mu_{-} (D_n - D)u}_{\cH}.
\end{equation}
The first term is estimated by $\norm{\mu_0 (D_n-D)}_{1,0} \norm{ u }_{\cW} $. The second summand is 
\[
\norm{\mu_{+} (D_n - D)u}_{\cH} \leq \norm{ (B_{+,n} - B_{+}) \mu_+ u}_{\cH} + \norm{\mu_{+} (\sigma_{+,n} - \sigma_+) \dot{u}}_{\cH} \leq 
\]
\[
\leq \norm{B_{+,n}-B_+}_{1,0}  \norm{u}_{\cW} + \norm{\sigma_{+,n} - \sigma_+} \norm{u}_{\cW} .
\]
The third summand in (\ref{estimate-for-nicolaescu-lemma}) is similarly estimated. So we can find a null sequence $c_n$ with $\| (D_n - D)u\|_{\cH} \leq c_n \norm{u}_{\cW}$ and the elliptic estimate proves that the third condition of Lemma \ref{nicolaescu-lemma} holds. Therefore $\normalize{D_x}$ depends continuously on $x$. By Lemma \ref{local-to-global-fredholm-families}, $x \mapsto \normalize{D_x}$ is a Fredholm family on the trivial Hilbert bundle $\cH_{X}$.
\end{proof}

We now drop the condition that all metrics are constant, in other words, we move to the space $\fF$. This is done by a ``gauging trick'', based on Observation \ref{gauging-lemma}.

\begin{proof}[Proof of Proposition \ref{fredholmproperty:familywise}, second and third step.]
Fix a metric $g_0$ on $N$ and a metric $h_0$ on $E$, both with cylindrical ends.
We will construct a map $\Phi: \fF  \to \fF_{g_0,h_0}$ which is covered by a Hilbert bundle isometry $\hat{\Phi}$ of the tautological Hilbert bundles that conjugates the families $(\normalize{D_x})_{x}$ over these spaces. We use Observation \ref{gauging-lemma} and define $\Phi$ by 
\begin{equation}\label{defn:phi-gauging}
\Phi:(g,h,\sigma,D)  \mapsto (g_0,h_0,\tau_{h_0,h}\sigma\tau_{h_0,h}^{-1},T_{(g,h),(g_0,h_0)} D T_{(g,h),(g_0,h_0)}^{-1}).
\end{equation}
The Hilbert space isometry is given by $T_{(g,h),(g_0,h_0)}: \cH_x \to \cH_{\Phi(x)}$. This proves that the tautological family over $\fF$ is the pullback of the tautological family over $\fF_{g_0,h_0}$, and as the latter is continuous by the first step of the proof, so is the former.

Now we have proven Proposition \ref{fredholmproperty:familywise} for the tautological family over $\fF$, i.e. for trivial manifold bundles. If $Q \to X$ is a $G$-principal bundle, and the manifold bundle is $Z=Q\times_G X$, we argue as follows. Being a continuous family of Fredholm operators is a local property. We may choose a local section $s$ of $Q \to X$ over $U \subset X$. Under this section, we can view the family of pseudo-differential operators as a map $U \to \fF$. This completes the proof.
\end{proof}

\subsection{Formulation of the generalized spectral-flow index theorem}\label{relevantexamples}

\begin{assumptions}\label{assumptions:familydirac}
Let $M$ be a closed manifold and let $X$ be a paracompact space. The projection map $\pi:X \times \bR \times M \to X \times \bR$ is an $M$-bundle, and the projection map $\Pi: X \times \bR \times M \to X$ is an $\bR\times M$-bundle, which has cylindrical ends. Let $E \to X \times \bR \times M$ be a graded (finite-dimensional) $\Cl^{p,q}$-Hilbert bundle. Let $g=(g_{(x,t)})_{(x,t) \in X \times \bR}$ be a family of Riemannian metrics on $M$. Finally, let $A=(A_{(x,t)})_{(x,t) \in X \times \bR}$ be a family of odd $\Cl^{p,q}$-linear pseudo-differential operators of order $1$ ($A_{(x,t)}$ is a pseudo-differential operator on the bundle $E_{(x,t)}:=E|_{(x,t) \times M}$).
We assume that all data are smooth in the $\bR \times M$-direction and cylindrical outside $\bI$, and that $A_{(x,t)}$ is invertible for $|t|\geq 1$.

There is a version of these data relative to a closed subspace $Y \subset X$. In that case, we assume that $g_{(y,t)}$ is independent of $t$ for $y \in Y$, that $E|_{Y \times \bR \times M}$ is of the form $\pr_{Y \times M}^{*} E_0$ (as a $\Cl^{p,q}$-Hilbert bundle) and that $A_{y,t}$ does not depend on $t$ if $y \in Y$ (note that in particular $A_{(y,t)}$ is invertible for all $(y,t) \in Y \times \bR$). 
\end{assumptions}

\begin{definition}
Let $M$, $X$, $Y$, $g$, $E$ and $A$ be as in Assumption \ref{assumptions:familydirac}. The family $(x,t) \mapsto \normalize{A_{(x,t)}}$ defines an element $\Lambda (E,A)\in \Omega F^{p,q} (X,Y)$ (by a simple application of Proposition \ref{fredholmproperty:familywise}). 
\end{definition}

So far, we considered the $M$-bundle $\pi$. We want to define a family of operators on the $\bR \times M$-bundle $\Pi$, and obtain an index in $F^{p+1,q}(X)$ (using the analysis contained in Proposition \ref{fredholmproperty:familywise}). For that, we need one more piece of datum.

\begin{assumptions}\label{assumptions:familydirac2}
Let $M$, $X$, $g$, $E$ and $A$ be as in Assumption \ref{assumptions:familydirac}. Consider $E$ as a vector bundle over the manifold bundle $\Pi$, and pick a $\Cl^{p,q}$-linear metric connection $\nabla$ on $E$ such that the grading involution $\iota$ is parallel and such that $\nabla$ is cylindrical outside $X \times \bI$ and over $Y \times \bR$ (this implies that outside $X \times \bI$ and over $Y \times \bR$, the operator $\nabla_{\partial_t}$ is the same as $\partial_t$). 
For $x \in X$, we obtain a vector bundle $E_x = E|_{x \times \bR \times M} \to \bR \times M$. The operators $A_{(x,t)}$ induce an operator $A_x$ on $\Gamma_c (\bR \times M;E_x)$ (which is not elliptic). Likewise, the connection $\nabla$ induces an operator $\nabla_{\partial_t}$ on $\Gamma_c (\bR \times M;E_x)$.
\end{assumptions}

\begin{definition}\label{defsuspension}
The \emph{suspension} of the data $(E,A,\nabla)$ is the following amount of data. Let $\Sigma E:= E \oplus E\to X \times \bR \times M$. We define a $\Cl^{p+1,q}$-action $c'$ and a grading $\iota'$ by
\[
\iota':= \twomatrix{\iota}{}{}{-\iota}, \; c'(e_1) :=\twomatrix{}{\iota}{-\iota}{}, \; c' (\eps_i) := \twomatrix{\eps_i}{}{}{\eps_i}, \; c' (e_i) := \twomatrix{e_{i-1}}{}{}{e_{i-1}} \, (i \geq 2).
\]
It is easy to see that
\begin{equation}
(D_A)_{x}:= \twomatrix{A_x}{- \nabla_{\partial_t}}{ \nabla_{\partial_t}}{- A_x}
\end{equation}
is a symmetric, odd and $\Cl^{p+1,q}$-linear pseudo-differential operator $(D_A)_x:
\Gamma_c (x \times \bR \times M; \Sigma E)\to \Gamma_c (x \times \bR \times M; \Sigma E)$ of order $1$. 
Since outside $X \times \bI$, $\nabla_{\partial_t}=\partial_t$, the operator $(D_A)_x$ is of the form that was considered in section \ref{elliptic-preliminaries} so that all the results now apply to this situation. In particular, we get a $\Cl^{p+1,q}$-Fredholm operator $(D_A)_x$ by Proposition \ref{fredholmproperty:pointwise}. 
Due to Proposition \ref{fredholmproperty:familywise}, we get a $\Cl^{p+1,q}$-Fredholm family
\[
x \mapsto \frac{(D_A)_x}{(1+ (D_A)_{x}^{2})^{1/2}}
\]
over $X$ and hence a well-defined element $\susp(E,A, \nabla) \in F^{p+1,q}(X)$. 
\end{definition}
We wish to have a relative index as well.

\begin{lemma}\label{spectralflow:basic-properties}\mbox{}
We use the notation of Assumptions \ref{assumptions:familydirac} and \ref{assumptions:familydirac2}. Then $\susp (E,A,\nabla)$ defines a relative $K$-theory class in $F^{p+1,q} (X,Y)$ which does not depend on the choice of $\nabla$. This $K$-theory class is denoted by $\susp (E,A)$.
\end{lemma}
\begin{proof}
There is a function $c : Y \to (0, \infty)$ such that $A_y^2 \geq c(y)$ holds for all $y \in Y$. Then 
\[
(D_A)_y^2= \twomatrix{A^2_y+ \nabla_{\partial_t}^{*}\nabla_{\partial_t}}{}{}{A^2_y+ \nabla_{\partial_t}^{*}\nabla_{\partial_t}} \geq c(y)
\]
(to get this simple formula for $D_A^2$, it is necessary that $A_{y,t}$ does not depend on $t$). Therefore 
\[
\left( \normalize{(D_A)_y}\right)^2 \geq \frac{c(y)^2}{1 + c(y)^2}.
\]
The function $y \mapsto \frac{c(y)^2}{1 + c(y)^2}$ is positive on $Y$, and thus $\normalize{(D_A)_y}$ is invertible, by Lemma \ref{invertibility-via-positivity}.

The independence of $\nabla$ is clear: the space of connections with the properties used in the construction is convex, and we built in concordance invariance into the definition of $F^{p,q}$.
\end{proof}

\begin{definition}\label{defn:pseudo-dirac}
A symmetric elliptic pseudo-differential operator $D$ of order $1$ on sections of a vector bundle $E \to M$ on a Riemannian manifold is called \emph{pseudo-Dirac operator} if the leading symbol $\symb_D$ is the leading symbol of a differential operator of Dirac type \cite[Def. 3.1]{BW}. In other words, we require that $\xi \mapsto \symb_D (\xi)$ is a linear map and for all $\xi \in T^* M$, the equation $\symb_D (\xi)^2 = + | \xi|^2$ holds. 
\end{definition}

We can now finally state the first version of the main result of this paper, which is a generalization of the spectral-flow index theorem. Recall the Bott map $\bott:F^{p+1,q} (X,Y)\to \Omega F^{p,q} (X,Y)$. 

\begin{theorem}\label{mainresult-homotopyversion}
Let $(X,Y)$ be a compact CW pair. Assume that $M$, $g$, $E$ and $A$ is as in Assumption \ref{assumptions:familydirac} and that in addition, $A_{(x,t)}$ is a pseudo-Dirac operator for all $(x,t )\in X \times \bR$. Then the two elements $\bott \circ \susp (E,A)$ and $\Lambda(E,A)$ in $\Omega F^{p,q} (X,Y)$ are equal.
\end{theorem}

The proof of Theorem \ref{mainresult-homotopyversion} will use a $K$-theoretic reformulation that we now develop.

\begin{definition}\label{def-l-group}
Let $(X,Y) $ be a pair of spaces and let $M$ be a closed Riemannian manifold. Consider the abelian group, generated by isomorphism classes of pairs $(V,A)$, where 

\begin{enumerate}
\item $V \to X \times M $ is a finite-dimensional $\Cl^{p,q}$-Hilbert bundle, smooth in $ M$-direction.
\item $A=(A_{(x,t)})_{(x,t) \in X \times \bR}$ is a smooth (in $t$-direction) family of $\Cl^{p,q}$-linear, odd pseudo-Dirac operators on $\bR \times V\to X \times \bR \times M$ such that $A_{(x,t)} = A_{(x,1)}$ for $t>1$, $A_{(x,t)} = A_{(x,-1)}$ for $t<-1$, $A_{(x,\pm 1 )}$ is invertible for each $x \in X$ and $A_{(y,t)}$ is $t$-independent for each $y\in Y$.
\end{enumerate}
There are obvious notions of direct sum, isomorphism and concordance of pairs $(V,A)$, and concordance is an equivalence relation on the set of isomorphism classes.
We consider the abelian group generated by the concordance classes of pairs $[V,A]$ and divide out the subgroup generated by the following relations.

\begin{enumerate}
\item $[V_0,A_0] + [V_1,A_1] = [ V_0 \oplus V_1, A_0 \oplus A_1]$.
\item $[V,A]=-[V,A']$, where $A'$ is the family $A'_{(x,t)} = A_{(x, -t)}$.
\end{enumerate}
The quotient group by this equivalence relation is denoted $L^{p,q}_{M}(X,Y)$.
\end{definition}

Given such a pair $(V,A)$, we have the bundle $\bR \times V \to X \times \bR \times M$ and 
we are in the situation of Assumptions \ref{assumptions:familydirac} and \ref{assumptions:familydirac2}. Therefore, using Lemma \ref{spectralflow:basic-properties}, we have elements 
\[
\susp(V,A):=\susp (\bR \times V,A) \in F^{p+1,q}(X,Y)\; \text{ and } \;\Lambda (V,A):= \Lambda (\bR \times V,A) \in \Omega F^{p,q} (X,Y).
\]

\begin{remark}\label{connection-is-parder}
The connection $\nabla$ on $\bR \times V$ used to define the suspension can be chosen so that $\nabla_{\partial_t} = \partial_t$. 
\end{remark}

\begin{proposition}\label{def-spectral-flow-map}
The classes $\susp(\bR \times V,A)\in  F^{p+1,q}(X,Y)$ and $\Lambda (\bR \times V,A) \in \Omega F^{p,q} (X,Y)$ only depend on the equivalence class $[V,A] \in L^{p,q}_{M} (X,Y)$.
\end{proposition}
\begin{proof} 
That both construction preserve direct sums is trivial. That they preserve concordances is also quite clear, using Proposition \ref{fredholmproperty:familywise} (since a homotopy defines a pair on the product bundle over $X \times I$). 
Reflection in the $\bR$-axis acts by multiplication with $-1$ on $\Omega F^{p,q}(X,Y)$. This follows from Theorems \ref{compact-comparison-theorem}, \ref{atiyah-singer-karoubi1} and \cite[Lemma 2.4.5]{At66}.
Therefore the $\Lambda$-construction preserves the reflection relation in $L^{p,q}_{M} (X,Y)$.
That the suspension preserves the reflection relation is less trivial to prove.
Let $\phi: \cH \to \cH$ be the involution induced by reflection $t \mapsto -t$, which preserves the Clifford structure and the grading. Conjugating $D_{A'}$ by $\phi$ gives the operator $\phi D_{A'} \phi$. The direct sum $D_A \oplus \phi D_{A'} \phi$ is
\[
F:=\begin{pmatrix}
A_t & - \partial_t & & \\
 \partial_t  &  -A_t & & \\
   &  & A_t &  \partial_t\\
   & &  -\partial_t & -A_t\\
\end{pmatrix}, \; \text{and we define}\;
Q:=\begin{pmatrix}
 & & & 1\\
  & &-1 & \\
   & -1 & & \\
  1  & & & \\
\end{pmatrix}.
\]
The operator $Q$ anticommutes with $F$, is $\Cl^{p+1,q}$-linear, self-adjoint, odd and satisfies $Q^2 =1$. This implies that $\normalize{F}$ is homotopic to an invertible operator, by Lemma \ref{nullhomotopy-criterion} (note that the normalizing function $\normalize{x}$ is odd, and thus $\normalize{F}$ anticommutes with $Q$).
\end{proof}

Therefore, we have two maps
\[
\bott \circ \susp, \, \Lambda: L^{p,q}_{M} (X,Y) \to \Omega F^{p,q} (X,Y).
\]
The $K$-theoretic version of Theorem \ref{mainresult-homotopyversion} is 

\begin{theorem}\label{main-index-theorem}
If the manifold $M$ is nonempty and of positive dimension, then the two maps $\bott \circ \susp$ and $\Lambda: L^{p,q}_{M} (X,Y) \to \Omega F^{p,q} (X,Y)$ agree for each finite CW pair $(X,Y)$.
\end{theorem}

It might seem that Theorem \ref{mainresult-homotopyversion} follows immediately from Theorem \ref{main-index-theorem}. However, the proof requires a nontrivial argument and is therefore deferred to section \ref{hard-gauging} below.

\subsection{The definitions of the index difference}\label{defnsinddiff}
Let $M^d$ be a closed $d$-dimensional spin manifold, with a Riemannian metric $g$. Let us briefly recall the spin package \cite[\S II.7]{LM}. 
The spinor bundle $\spinor_g \to M$ is a finite-dimensional, fiberwise irreducible $\Cl^{TM}$-$\Cl^{d,0}$-Hilbert bimodule bundle. It carries a canonical connection $\nabla$ inherited from the Levi-Civita connection on $M$. Using Clifford multiplication and the connection, one defines the Atiyah-Singer-Dirac operator $\Dir_g$ as the composition
\[
\Gamma (M;\spinor_g) \stackrel{\nabla}{\to} \Gamma (M; T^* \otimes \spinor_g) \stackrel{c}{\to} \Gamma (M;\spinor_g);
\]
this is a $\Cl^{d,0}$-linear, odd and symmetric elliptic differential operator. It is related to scalar curvature by the Schr\"odinger-Lichnerowicz formula

\begin{equation}\label{weitzenboeck}
\Dir_g^2 = \nabla^* \nabla + \frac{1}{4} \scal_g.
\end{equation}

We wish to give the precise definitions of the index difference and to this end, we produce the data as described in Assumption \ref{assumptions:familydirac}.
Let $\Riem (M)$ be the space of Riemannian metrics (an open convex subspace of the Fr\'echet space of smooth symmetric bilinear forms on the tangent bundle of $M$) and $\Riem^+ (M) \subset \Riem (M)$ be the subspace of metrics with positive scalar curvature. We put $X= \Riem^+ (M) \times \Riem^+ (M)$ and let $Y =\Delta \subset X$ be the diagonal.
Pick a smooth function $a: \bR \to [0,1]$ that is $1$ on $(-\infty,-1]$ and $0$ on $[1,\infty)$. Let $x=(x_{-1},x_1) \in \Riem^+ (M) \times \Riem^+ (M)$ and $t \in \bR$. Let
\[
g_{(x,t)} := a(t)  x_{-1} + (1-a(t)) x_{1}.
\]
If $|t|\geq 1$, the scalar curvature of $g_{(x,t)}$ is positive; likewise, if $x \in Y$ (i.e. $x_{-1}=x_1$), then $g_{(x,t)}$ does not depend on $t$ and has positive scalar curvature. This construction yields a fiberwise metric on the $M$-bundle $\pi: X\times \bR \times M \to X \times \bR$. Let $\spinor^d \to  X\times \bR \times M $ be the fiberwise spinor bundle (the superscript should remind the reader that it is a $\Cl^{d,0}$-bundle). On the fiber $\pi^{-1}(x,t)$, there is the Atiyah-Singer-Dirac operator $\Dir_{g_{(x,t)}}$, which is $\Cl^{d,0}$-linear and yields a family $\Dir^d$ of elliptic operators over $X \times \bR$. By construction, all these data are cylindrical outside $X \times \bI$ and constant in $\bR$-direction over $Y$. Therefore, we are in the situation of Assumption \ref{assumptions:familydirac} and the connection $\nabla$ satisfies the properties of Assumption \ref{assumptions:familydirac2}. By the formula (\ref{weitzenboeck}), $\Dir_{g_{(x,t)}}$ is invertible whenever $|t| \geq 1$ or $x \
in Y$. 
All the assumptions of our analysis are satisfied, and we define Hitchin's version of the index difference as 
\[
\inddiffH := \Lambda (\spinor^d,\Dir^d) \in  \Omega F^{d,0} (X,Y).
\]
Since $(X,Y)$ is a paracompact and compactly generated pair, we can invoke Theorem \ref{paracompact-comparison-theorem} and obtain a well-defined homotopy class
\[
\inddiffH: (\Riem^+ (M) \times \Riem^+ (M), \Delta) \to (\Omega K^{d,0}, \Omega D^{d,0}).
\]

To give Gromov and Lawson's definition of the index difference, we consider the trivial $\bR \times M$-bundle 
\[
\Pi:  \Riem^+ (M) \times \Riem^+ (M)\times \bR \times M\to \Riem^+ (M) \times \Riem^+ (M).
\]
We equip the fiber over $x \in  \Riem^+ (M) \times \Riem^+ (M)$ with the metric $dt^2 + g_{(x,t)} $. There is the spinor bundle $\spinor^{d+1}$ of the $\bR \times M $-bundle $\Pi$ (with the above metric), with the family of Dirac operators $\Dir^{d+1}$. The suspension of $\spinor^d$ is $\spinor^{d+1}$. This of course depends on a convention how a spin structure on $M$ induces a spin structure on the cylinder $\bR \times M$, but we do not spell out the convention here.
The operators $\Dir^{d+1}$ and $D_{\Dir^d}$ are not quite the same, but they have the same symbol and they agree over $Y \times \bR$ and outside $X \times \bI$. Therefore, the Fredholm families defined by $\Dir^{d+1}$ and $D_{\Dir^d}$ are concordant over $(X,Y)$ and so they define the same element in $F^{d+1,0} (X,Y)$.
We define Gromov and Lawson's version of the index difference as
\[
\inddiffGL:= \susp (\spinor^d, \Dir^d) \in F^{d+1,0} (X,Y)
\]
and get a map 
\[
\inddiffGL: (\Riem^+ (M) \times \Riem^+ (M),\Delta) \to (K^{d+1,0},D^{d+1,0}).
\]

If $(X,Y)$ is a finite CW pair and $f:(X,Y) \to (\Riem^+ (M) \times \Riem^+ (M), \Delta)$, then $\inddiffH \circ f \sim \bott \circ \inddiffGL \circ f$ (relative homotopy) by Theorem \ref{mainresult-homotopyversion}, and so we have proven Theorem \ref{mainresult-psc}.

\section{Proof of the main result}\label{ktheory}

In this section, we give the proof of Theorem \ref{main-index-theorem} which follows a common pattern in index theory: we use $K$-theoretic arguments to reduce the problem to a single index computation, and this formal idea is the same as in the classical papers \cite{AS1,At}. Recall that Theorem \ref{main-index-theorem} asserts the commutativity of the diagram
\[
\xymatrix{
L^{p,q}_{M} (X,Y) \ar[r]^{\susp} \ar[dr]_{\Lambda} & F^{p+1,q} (X,Y) \ar[d]^{\bott}\\
 & \Omega F^{p,q} (X,Y).
}
\]

The most difficult step is the following result.

\begin{theorem}\label{isomorphism-theorem}
If $M$ is nonempty and of positive dimension and $(X,Y)$ a CW pair, the map $\Lambda: L^{p,q}_{M} (X,Y) \to \Omega F^{p,q} (X,Y)$ is an isomorphism.
\end{theorem}

This is proven in section \ref{section:invertiblepseudos} and relies on the richness of the class of pseudo-differential operators. An interesting aspect is that $L^{p,q}_{M}$ does not depend on $M$. In section \ref{formal-structures}, we prove certain formal properties of the groups $L^{p,q}:=L^{p,q}_{M}$ and the maps $\susp$ and $\Lambda$; these formal properties are then used to reduce everything to the simple special case $(p,q)=(0,1)$ and $(X,Y)=(*,\emptyset)$. We give an explicit computation for this case, alternatively, one could use the classical spectral-flow-index theorem as proven by Robbin and Salamon \cite{RoSa}. 

\subsection{Spaces of invertible pseudo-differential operators with given symbol}\label{section:invertiblepseudos}

Here we prove Theorem \ref{isomorphism-theorem}, which is essentially a result on spaces of invertible pseudo-differential operators. Throughout this section, we fix a Clifford index $(p,q)$ (the case $(p,q)=(0,1)$ is what is really important for us).

\begin{definition}
Let $M$ be a closed Riemannian manifold and $V \to M$ be a finite-dimensional $\Cl^{p,q}$-Hilbert bundle. For $m \in \bZ$, let $\pseudop{\Cl}{m}(V)\subset \pseudo{m}(V)$ be the space of order $m$, $\Cl^{p,q}$-linear pseudo-differential operators acting on $V$. Let $\pseudop{\Cl}{m}(V)_{sa,odd} \subset \pseudop{\Cl}{m}(V)$ be the subspace of symmetric and odd elements (a similar national convention is used for all $\bZ/2$-graded $*$-algebras in the sequel). These spaces carry the Fr\'echet space topology inherited from $\pseudo{m}(V)$.
\end{definition}

Assume that $V$ has a $\Cl(TM)$-left module structure, so that pseudo-Dirac operators on $V$ exist. Let $\pseudir{V}\subset \pseudop{\Cl}{1}(V)_{sa,odd}$ be the space of ($\Cl^{p,q}$-linear, odd) pseudo-Dirac operators and let $\pseudir{V}^{\times}$ be the open subspace of invertible operators. We assume that $\pseudir{V}^{\times} \neq \emptyset$ and pick a basepoint $B \in \pseudir{V}^{\times}$. Then $\pseudir{V}$ is the affine subspace $B + \pseudop{\Cl}{0}(V)_{sa,odd} \subset \pseudop{\Cl}{1}(V)_{sa,odd}$. Let 
\[
\fX:=\Lopspace
\]
be the space of all smooth families $A: \bI \to \pseudir{V}$ such that $A( 1 )$ is invertible and $A(-1)=B$. There is a map 
\[
\Theta: \fX \to \Omega \Fred^{p,q} (L^2 (M;V)), \; \Theta (A) := (t \mapsto \normalize{A(t)});
\]
one uses Proposition \ref{fredholmproperty:familywise} to show that $t \mapsto \normalize{A(t)}$ defines an element in $\Omega \Fred^{p,q} (L^2 (M;V))$ and that $\Theta$ is continuous.
The analytical core of Theorem \ref{isomorphism-theorem} is the following result.

\begin{proposition}\label{thm:invertiblepseudos}
Let $M$ and $V$ as above. Assume that the Hilbert space $H:=L^2 (M;V)$ is ample and that there exists an invertible $\Cl^{p,q}$-linear, odd pseudo-Dirac operator $B$ on $V$.
Then the map $\Theta$ is a weak homotopy equivalence.
\end{proposition}

The idea for the proof of Proposition \ref{thm:invertiblepseudos} that we give is borrowed from Boo{\ss}-Wojciechowski \cite[\S 15]{BW}.
We will use the following general result on homotopy types of open subsets of topological vector spaces, due to Palais.

\begin{theorem}\label{palais2} \cite[Corollary to Theorem 12]{Pal2}
Let $f:V_0\to  V_1$ be a continuous linear map between locally convex topological vector spaces. Assume that $f$ has dense image. Let $U \subset V_1$ be open. Then $f|_{f^{-1}(U)}: f^{-1}(U) \to U$ is a weak homotopy equivalence.
\end{theorem}
The linear structure in Theorem \ref{palais2} is essential, and the first step is to replace $A \mapsto \normalize{A}$ by a linear map.
Consider $S:= (1+B^2)^{-1/2}$, which is an invertible pseudo-differential operator of order $-1$. 

\begin{lemma}\label{interpolation-lemma}
The map $\Theta':\fX \to \Omega \Fred^{p,q}(H)$, $A \mapsto (t \mapsto S^{1/2}A(t) S^{1/2})$ is continuous and homotopic to $\Theta$.
\end{lemma}
\begin{proof}
The map $\pseudo{1}(V) \to \Lin (W^{1/2}, W^{-1/2})$ that takes a pseudo-differential operator to its induced map on Sobolev spaces is continuous by Proposition \ref{pseudo-properties}. On the other hand, the estimate

\[
\norm{S^{1/2}A S^{1/2}}_{0,0} \leq \norm{S^{1/2}}_{-1/2,0}\norm{A}_{1/2,-1/2}\norm{S^{1/2}}_{0,1/2}
\]
shows that the map $\Lin (W^{1/2}, W^{-1/2}) \to \Lin (W^{0}, W^{0})$, $A \mapsto S^{1/2}A S^{1/2}$ is continuous as well. Therefore, $\Theta'$ is continuous. A homotopy between $\Theta$ and $\Theta'$ is given by
\[
\Theta_s: A  \mapsto \left( t \mapsto (1 + s B^2 + (1-s) A(t)^2)^{-1/4}  A(t)  (1 + s B + (1-s) A(t)^2)^{-1/4}\right).\qedhere
\]
\end{proof}

For the rest of the proof of Proposition \ref{thm:invertiblepseudos}, we work with $\Theta'$. Recall that we fixed a basepoint $B  \in \pseudir{V}^{\times}$ and we let $B_0:=S^{1/2} B S^{1/2}=\normalize{B}$. Let $\Kom_{\Cl} (H)_{sa,odd} \subset \Kom (H)$ be the space of self-adjoint, odd, Clifford-linear compact operators on $H$ and $(B_0 + \Kom_{\Cl} (H)_{sa,odd})^{\times}$ be the space of all invertible (self-adjoint, odd, Clifford-linear) operators on $H$ that differ from $B_0$ by a compact operator. Observe that we get a map 
\[
\varphi:\pseudir{V}^{\times} \to (B_0 + \Kom_{\Cl} (H)_{sa,odd})^{\times}, \; A \mapsto S^{1/2} A S^{1/2}
\]
(by Rellich's theorem), which is continuous by the argument given in the proof of Lemma \ref{interpolation-lemma}. The next step is to prove:

\begin{proposition}\label{approximation-lemma}
The map $\varphi$ is a weak homotopy equivalence.
\end{proposition}
\begin{proof}
Consider the continuous linear map $\theta:\pseudop{\Cl}{0}(V)_{sa,odd} \to \Kom_{\Cl} (H)_{sa,odd}$, $\theta (P):= S^{1/2} P S^{1/2}$. Let $U \subset \Kom_{\Cl} (H)_{sa,odd}$ be the open subset of all $Q$ such that $B_0 + Q$ is invertible. There is a commutative diagram

\[
\xymatrix{
\theta^{-1}(U) \ar[r]^{\theta|_{\theta^{-1}(U)}} \ar[d]^{P \mapsto B + P} & U \ar[d]^{Q \mapsto B_0 + Q} \\
\pseudir{V}^{\times} \ar[r]^-{\varphi} & (B_0 + \Kom_{\Cl} (H)_{sa,odd})^{\times};
}
\]
the vertical maps are homeomorphisms. Proposition \ref{approximation-lemma} follows if we can show that $\theta^{-1}(U) \to U$ is a weak homotopy equivalence, and this follows from Theorem \ref{palais2} once we know that $\theta$ has dense image. The map $\theta$ factors as $\pseudop{\Cl}{0}(V)_{sa,odd} \to \pseudop{\Cl}{-1}(V)_{sa,odd} \to \Kom_{\Cl}(H)_{sa,odd}$, the first is $P \mapsto S^{1/2}  P S^{1/2} $ and a homeomorphism, and so we have to prove that $\pseudop{\Cl}{-1}(V)_{sa,odd} \to \Kom_{\Cl}(H)_{sa,odd}$ has dense image.
It is a standard result that infinitely smoothing operators are dense in the space of all compact operators (without the conditions ``self-adjoint, $\Cl^{p,q}$-linear, odd'') and so the image of $\pseudo{0}(V) \to \Kom (H)$ is dense: one picks an orthonormal basis $(e_n)$ in $H=L^2 (M;V)$ consisting of smooth sections. The operators $u \mapsto \langle u, e_m\rangle e_n$ are infinitely smoothing and they span a dense subspace of the compact operators. To see that one can built in the algebraic conditions (``self-adjoint, odd and $\Cl^{p,q}$-linear''), let $\Gamma \subset \Cl^{p,q}$ be the multiplicative subgroup generated by all $e_i, \eps_j$. Then the formula 

\begin{equation}\label{averaging-operator}
K \mapsto  \frac{1}{4|\Gamma|} \left( \sum_{\gamma \in \Gamma} \gamma K \gamma^{-1} -  \iota  \gamma K \gamma^{-1} \iota \right) +  \frac{1}{4|\Gamma|}\left( \sum_{\gamma \in \Gamma} \gamma K \gamma^{-1} -  \iota  \gamma K \gamma^{-1} \iota \right)^*
\end{equation}
defines a projection $\Kom (H) \to \Kom_{\Cl} (H)_{sa,odd}$ which maps the image of $\pseudo{-1}(V)$ to the image of $\pseudop{\Cl}{-1}(V)_{sa,odd}$.
\end{proof}

\begin{corollary}\label{approximation-lemma3}
The map 
\[
\fX \to \map ( (\bI, \partial \bI, -1); (B_0 + \Kom_{\Cl} (H)_{sa,odd},(B_0 + \Kom_{\Cl} (H)_{sa,odd})^{\times},B_0)),
\]
defined by $A \mapsto (t \mapsto S^{1/2} A(t) S^{1/2})$, is a weak homotopy equivalence.
\end{corollary}
\begin{proof}
The spaces $B_0 + \Kom_{\Cl} (H)_{sa,odd}$ and $\pseudir{V}$ are convex, and by Proposition \ref{approximation-lemma}, the map between the spaces of invertibles is a weak homotopy equivalence. Hence 
\[
(\pseudir{V}, \pseudir{V}^{\times}, B) \to (B_0 + \Kom_{\Cl} (H)_{sa,odd},(B_0 + \Kom_{\Cl} (H)_{sa,odd})^{\times}, B_0)  
\]
induces a weak homotopy equivalence on mapping spaces.
By smooth approximation and another application of Theorem \ref{palais2}, one shows that the inclusion 
\[
\fX=\map_{C^{\infty}} ((\bI, \partial \bI, -1); (\pseudir{V}, \pseudir{V}^{\times}, B) ) \to \map ((\bI, \partial \bI, -1); (\pseudir{V}, \pseudir{V}^{\times}, B))
\]
is a weak homotopy equivalence.
\end{proof}

The third and last step of the proof of Proposition \ref{thm:invertiblepseudos} is purely topological, and needs another lemma.

\begin{lemma}\label{fibre-lemma}
Let 
\[
\xymatrix{
L \ar[d]^{j} \ar[r]^{\subset} & K \ar[d]\\
G \ar[r]^{\subset} \ar[d] & F \ar[d]^{q}\\
C \ar@{=}[r] & C
}
\]
be a commutative diagram of fibrations, and let $y \in L$ be a basepoint. Assume that $G$ and $K$ are contractible. Then the map 
\[
\map( (I, \partial I, 0) ; (K,L,y)) \to \map ((I, \partial I); (F,G)) \simeq \Omega F
\]
induced by the inclusion $K \subset F$ is a weak homotopy equivalence.
\end{lemma}
\begin{proof}
The map factors as 
\[
\map( (I, \partial I, 0) ; (K,L,y)) \to \map ((I, \partial I,0); (F,G,j(y))) \to \map ((I, \partial I); (F,G)) ,
\]
and the second map is a weak equivalence since $G$ is contractible. The effect of the first map on $\pi_k$ is easily identified with the map $\pi_{k+1} (K,L,y) \to \pi_{k+1}(F,G,j(y))$ induced by the inclusion. This latter map is an isomorphism by \cite[Proposition 6.3.8]{tD}.
\end{proof}

\begin{proof}[Proof of Proposition \ref{thm:invertiblepseudos}]
One applies the previous lemma. Let $F\subset \Fred^{p,q} (H)$ be the component containing $\cG^{p,q}$, $G =\cG^{p,q}$. As we assumed that $H$ is ample, $G$ is contractible by Lemma \ref{kuipertheorem}.
Let $C$ be the image of $F$ under the quotient map $q:\Lin_{\Cl} (H)_{sa,odd} \to \Cal_{\Cl}(H)_{sa,odd}:= (\Lin_{\Cl}(H)/\Kom_{\Cl}(H))_{sa,odd}$ to the Calkin algebra. We define $y:= B_0 \in G$. Then the fibers over $q(y) \in C$ are $L= (B_0 + \Kom_{\Cl} (H)_{sa,odd})^{\times} \subset K = (B_0 + \Kom_{\Cl} (H)_{sa,odd})$, respectively. 
The quotient map $q: F \to C$ and the restriction $q: G \to C$ are Serre fibrations. This is a folklore result, stated and proven as Proposition \ref{fibretheorem} below.
Invoking Corollary \ref{approximation-lemma3} and Lemma \ref{fibre-lemma} finishes the proof.
\end{proof}

We now turn to the proof of Theorem \ref{isomorphism-theorem}. The first step is to show that here exist enough $\Cl^{p,q}$-bundles on $M$ so that the statement has a chance to be true.

\begin{lemma}\label{existence-ample-bundle}
Assume that $M \neq \emptyset$ is a closed Riemannian manifold of positive dimension. Then there exists a $\Cl^{p,q}$-bundle $V \to M$ such the Hilbert space $L^2 (M;V)$ is ample and such that there is an invertible, $\Cl^{p,q}$-linear, odd pseudo-Dirac operator on $V$.
\end{lemma}
\begin{proof}
Start with any nonzero $\Cl^{0,0}$-Dirac bundle $V_0 \to M$, for example, one could take $V_0 := \Lambda^* T^* M$. Let $N_0$ be a finite-dimensional graded $\Cl^{p,q}$-module that contains each of the finitely many irreducible graded real $\Cl^{p,q}$-modules and consider the $\Cl^{p,q}$-Dirac bundle $V_1:=V_0 \grotimes N_0$ (graded tensor product). The $\Cl^{p,q}$-Hilbert space $L^2 (M;V_1)$ is then ample. Consider $V = V_1 \oplus V_{1}^{op}$.

There exists a $\Cl^{p,q}$-linear odd pseudo-Dirac operator $A$ on $V$. For example, one can take an arbitrary Dirac operator on $V_1$ and apply the operation (\ref{averaging-operator}) to obtain a $\Cl^{p,q}$-linear and odd operator $A_1$ and define $A := A_1 \oplus A_1$. By Lemma \ref{additiveinverse}, $A$ has index zero in $\KO^{q-p}(X)$. Therefore, there exists a compact operator $R$ such that $A+R$ is invertible. Among the compact operators, those with a smooth integral kernel lie dense and thus we can pick $R$ to be a smoothing operator. Then $A+R$ is the desired invertible pseudo-Dirac operator.
\end{proof}

\begin{proof}[Proof of Theorem \ref{isomorphism-theorem}]
First we prove surjectivity. Any element in $\Omega F^{p,q}(X,Y)$ can be represented as the class $[H,F]$, where $H$ is an ample $\Cl^{p,q}$-Hilbert space and $F: X \to \Omega \Fred^{p,q}(H)$ a map which maps $Y$ to constant paths of invertible operators. 
By Lemma \ref{existence-ample-bundle}, there exists an ample $\Cl^{p,q} $-bundle $V \to M$ (i.e. $L^2 (M;V)$ is ample) and an invertible pseudo-Dirac operator $B$ on $V$.
It follows that $H \cong  L^2 (M;V)$, since ample $\Cl^{p,q}$-linear Hilbert spaces are unique up to isomorphism. Thus $F$ can be assumed to be a map $X \to \Omega \Fred^{p,q} (L^2(M;V))$ that sends $Y$ to paths of invertible operators. By Proposition \ref{thm:invertiblepseudos}, $F$ is homotopic to a family of smooth curves of pseudo-Dirac operators beginning at $B$ and ending with an invertible one (here we used that $(X,Y)$ is a CW pair, as the map in Proposition \ref{thm:invertiblepseudos} is only a weak homotopy equivalence). This proves surjectivity.

Injectivity is by a similar argument. Let $\sum_i a_i [V_i,A_i] \in L^{p,q}_{M}(X,Y)$ be an element in the kernel of $\Lambda$. Due to the relation $[V,A]=-[V,A']$ in Definition \ref{def-l-group}, we can assume that all coefficients $a_i\in \bZ$ are positive. By the direct sum relation, we can represent the element by a single $[V,A]$. If $V$ is not ample, we pick an ample bundle $V'$ and the constant invertible family $A'$ (Lemma \ref{existence-ample-bundle}). Then $[V',A']=0$ and the bundle $V \oplus V'$ is ample. Thus, without loss of generality, the element is $[V,A]$ with an ample $V$. The statement that $\Lambda [V,A]=0$ means that the family $\Lambda(A)$ of Fredholm operators is homotopic to an invertible family, by Theorem \ref{compact-comparison-theorem}. Again, by Proposition \ref{thm:invertiblepseudos}, we can lift this homotopy to a homotopy of pseudo-Dirac operators and this means, by the homotopy relation in $L^{p,q}_{M}(X,Y)$, that $[V,A]=0$.
\end{proof}

\subsection{From Theorem \ref{main-index-theorem} to Theorem \ref{mainresult-homotopyversion}: Deformation of symbols}\label{hard-gauging}

Let $M$ be a closed manifold and let $(X,Y)$ be a finite CW pair. Let $E \to X \times \bR \times M$ be a $\Cl^{p,q}$-Hilbert bundle, which is cylindrical over $Y \times \bR$ and outside $X \times \bI$. In the situation of Theorem \ref{mainresult-homotopyversion}, we have a family $g_{(x,t)}$ of Riemannian metrics on $M$ and a family $A_{(x,t)}$ of pseudo-Dirac operators, subject to the conditions from Assumption \ref{assumptions:familydirac}. We can apply the gauging trick as in the second step of the proof of Proposition \ref{fredholmproperty:familywise} to make the bundle $E$ constant in $\bR$-direction (as a $\Cl^{p,q}$-Hilbert bundle). Thus, we can assume that $E= \bR \times V$, for some $\Cl^{p,q}$-bundle $V \to X \times M$.  
If the Riemannian metric $g_{(x,t)}$ were independent of $(x,t)$, we would be in the situation of Theorem \ref{main-index-theorem} and could conclude that $\bott (\susp (E,A))= \Lambda (E,A)$, as required.

However, Theorem \ref{mainresult-homotopyversion} does not assume that $g_{(x,t)}$ does not depend on $(x,t)$, and for the application to positive scalar curvature, it is of course essential to allow for varying metrics. We could try to apply the gauging trick to the metrics $g_{(x,t)}$, as in the proof of Proposition \ref{fredholmproperty:familywise}, and turn them into the Riemannian metric $g_0$, which does not depend on $(x,t)$. However, changing the operator family as in (\ref{defn:phi-gauging}) leaves the symbol of $A_{(x,t)}$ invariant. Therefore, the transformed operator is no longer a pseudo-Dirac operator, because the square of the leading symbol is $\xi \mapsto g_{(x,t)} (\xi, \xi)$, and this is \emph{different} from $g_0(\xi, \xi)$. However, we have achieved something, and summarize our progress.

\begin{proposition}\label{wrong-gauging}
Let $X$, $Y$, $M$, $E=\bR \times V$ and $(g_{(x,t)})_{(x,t)}$ be as in Assumption \ref{assumptions:familydirac}. Let $(A_{(x,t)})_{(x,t)}$ be a family of pseudo-Dirac operators, satisfying Assumption \ref{assumptions:familydirac}. Let $g_0$ be a fixed Riemannian metric on $M$. Then there exists a family of pseudo-differential operators $(B_{(x,t)})_{(x,t)}$ such that
\begin{enumerate}
\item The family $(B_{(x,t)})$ satisfies Assumption \ref{assumptions:familydirac} with respect to the metric $g_0$.
\item $\symb_{B_{(x,t)}} =\symb_{A_{(x,t)}}$, in particular $\symb_{B_{(x,t)}} (\xi)^2 = g_{(x,t)} (\xi, \xi)$.
\item $\Lambda (V,B) = \Lambda (V,A) \in \Omega F^{p,q} (X,Y)$ and $\susp (V,B) = \susp (V,A) \in F^{p+1,q} (X,Y)$.
\end{enumerate}
\end{proposition}

The remaining task is to show that we can deform the operator family $B$ back into a family of pseudo-Dirac operators, without changing either index. 

\begin{proposition}
Assume the notations of Proposition \ref{wrong-gauging}. Then there exists a family $B_{(s,x,t)}$, $(s,x,t)\in [0,1] \times X \times \bR$, satisfying Assumption \ref{assumptions:familydirac} with base space pair $([0,1] \times X, [0,1] \times Y)$ such that $B_{(0,x,t)} = B_{(x,t)}$ and such that $B_{(1,x,t)}$ is a pseudo-Dirac operator. 
\end{proposition}

Writing $(B_s)_{(x,t)}:= B_{(s,x,t)}$, it then follows that $\Lambda (V, B_0)=\Lambda(V,B_1)$ and $\susp (V,B_0)=\susp (V,B_1)$, as claimed.

\begin{proof}
As all operators are independent of $t$ for $|t|>1$, we restrict our attention to $X \times \bI$. Let 
\[
X \times \bI \stackrel{q}{\to} Z:= X \times \bI/\sim,\; (x,t) \sim (x',t') :\Leftrightarrow x=x' \in Y. 
\]
The initial family $B_{(x,t)}$ is pulled back via $q$, and we shall construct the family $B_{(s,x,t)}$ over $Z$ pull it back using $q$, to guarantee $t$-independence over $Y$. We denote points in $Z$ by $(x,t)$. We write $Z_0:=:=(Y \times \bI \cup X \times \partial \bI)/\sim$.

We first construct the symbol $b{(s,x,t)}$ of $B_{(s,x,t)}$ and then show how to lift the deformation of symbols to a deformation of operators.
Let $h_{(x,t)}$ be the $g_0$-positive definite endomorphism of $T^* M$ such that
\[
g_{(x,t)} (h_{(x,t)} \xi, h_{(x,t)}\xi)= g_0 ( \xi,\xi)
\]
holds for all $\xi \in T^* M$. As explained in Observation \ref{gauging-lemma}, we can pick $h_{(x,t)}$ to depend smoothly on $t$ and continuously on $x$. Let 
\[
b_{(s,x,t)} (\xi) := \symb_{B_{(x,t)}} (((1-s)+ s h_{(x,t)}) \xi), \; (s,x,t) \in [0,1] \times Z
\]
for unit cotangent vectors $\xi$ of length $1$. Clearly, $b_{(s,x,t)}$ is an elliptic symbol which is $\Cl^{p,q}$-linear, self-adjoint and odd.

Pick a basepoint $o \in X$ and let $V_o:= V|_{o\times M}$. Let $H$ be the topological group of all even, $\Cl^{p,q}$-linear isometries of $V_o$. Then there is an $H$-principal bundle $P \to Z$, constant in $\bI$-direction and an isomorphism $V \cong P \times_H V_o$ of bundles over $Z$.

Let $\Symb_0 (V_o)$ be the $*$-algebra of Clifford-linear smooth symbols of order $0$ on $V$. More precisely, it is the algebra of all smooth sections of the bundle $\pi^* \End_{\Cl} (V) \to  ST^* M $, the bundle of $\Cl^{p,q}$-linear endomorphisms of $V$, pulled back to the unit cotangent bundle of $M$. Let $\overline{\Symb}_0 (V_o)$ be the closure under the maximum norm; this is a $C^*$-algebra. It is a graded $C^*$-algebra, by the even/odd-grading and the maximum norm.

The closure $\Pseudop{\Cl}{0}(V_o)$ under the $\norm{\_}_{0,0}$-operator norm is a graded $C^*$-algebra. By \cite[(5.2)]{AS1}, the symbol map $\Pseudop{\Cl}{0}(V_o) \to \overline{\Symb}_0 (V_o)$ is a surjective $*$-homomorphism. The group $H$ acts on both algebras and we get the bundle map
\[
P \times_H \Pseudop{\Cl}{0}(V_o) \to P \times_H \overline{\Symb}_0 (V_o)
\]
over $Z$. The induced map 
\begin{equation}\label{symbol-map}
\Gamma(Z;P \times_H \Pseudop{\Cl}{0}(V_o)) \to \Gamma(Z; P \times_H \overline{\Symb}_0 (V_o))
\end{equation}
is a surjective homomorphism of $C^*$-algebras (with the maximum norm). By Proposition \ref{fibretheorem}, the restriction 
\[
\Gamma(Z_0;P \times_H \Pseudop{\Cl}{0}(V_o))_{sa,odd}^{\times} \to \Gamma(Z_0; P \times_H \overline{\Symb}_0 (V_o))_{sa,odd}^{\times}
\]
to the subspaces of odd, self-adjoint invertible elements is a fibration. Pick a $\Cl^{p,q}$-linear metric even connection $\nabla$ on $V$.
The family $C_{(0,x,t)} := (1+\nabla^* \nabla)^{-1/4} B_{(x,t)}  (1+\nabla^* \nabla)^{-1/4}$ is a section of $\Gamma(Z;P \times_H \Pseudop{\Cl}{0}(V_o))_{sa,ev}$, with $\symb_{C_{(0,x,t)}} (\xi)= b_{(0,x,t)}(\xi)$ for all unit vectors $\xi$, while $s \mapsto b_{(s,x,t)}$ is a homotopy of sections of $\Gamma(Z; P \times_H \overline{\Symb}_0 (V_o))_{sa,ev}$. Over $Z_0$, $C_{(0,x,t)}$ is invertible. Thus we can lift the homotopy of symbols $b_{(s,x,t)}$ to a homotopy $C_{(s,y,t)}$ of invertible operators of order $0$ (for $(x,t) \in Z_0$), beginning with the invertible operator $C_{(0,x,t)}$. This lift can be extended to a lift of the symbols to a (non-invertible) family of operators $C_{(s,x,t)}$. The operator $C_{(s,x,t)}$ is not yet a pseudo-differential operator, but lies in the closure with respect to the $\norm{\_}_{0,0}$-norm. To overcome this problem, note that there exists a family $D_{(s,x,t)}$ of pseudo-differential operators with symbol $b_{(s,x,t)}$ (but without the invertibility condition). To 
construct such a family, one argues as follows. Because the symbol map $\pseudop{\Cl}{m}(V_o) \to \Symb_0 (V_o)$ is a surjective map of Fr\'echet spaces, it has a local section, by \cite[Theorem 10]{Pal2}. It then follows by an argument with a partition of unity that the map (\ref{symbol-map}), but without taking closures, is surjective, and this proves the existence of the family $D_{(s,x,t)}$. 
The family $E_{(s,x,t)}:=C_{(s,x,t)} - D_{(s,x,t)}$ lies in $\Gamma (Z; P \times_H \Pseudop{\Cl}{0}(V_o)_{sa,ev})$, and after adding $D_{(s,x,t)}$, it becomes invertible over $Z_0$ (which is an open condition). There are arbitrarily close approximations of $E_{(s,x,t)}$ by families of pseudo-differential operators $\tilde{E}_{(s,x,t)}$.
Choose the approximation close enough to guarantee that $\tilde{C}_{(s,x,t)}:=\tilde{E}_{(s,x,t)}+ D_{(s,x,t)}$ is invertible for $(x,t)\in Z$ and furthermore, one can pick $\tilde{E}_{(0,x,t)}=E_{(0,x,t)}$.
To get the desired operator $B_{(s,x,t)}$, put 
\[
B_{(s,x,t)} := (1+\nabla^* \nabla)^{1/4} \tilde{C}_{(s,x,t)} (1+\nabla^* \nabla)^{1/4}.\qedhere
\]
\end{proof}

\subsection{Proof of the main result - formal structures}\label{formal-structures}

From now on, we identify all groups $L^{p,q}_{M}$ using the isomorphism of Theorem \ref{isomorphism-theorem} and drop the subscript.

\begin{definition}
Let $\cP_{(p,q)} (X,Y)$ be the statement that the conclusion of Theorem holds for the finite CW pair $(X,Y)$ and the pair of numbers $(p,q)$, i.e. that the diagram
\[
\xymatrix{
L^{p,q} (X,Y) \ar[r]^-{\susp} \ar[dr]_{\Lambda} & F^{p+1,q} (X,Y) \ar[d]^{\bott}\\
 & \Omega F^{p,q} (X,Y)
}
\]
commutes and let $\cP_{(p,q)}$ be the statement that $\cP_{(p,q)} (X,Y)$ holds for all pairs $(X,Y)$.
\end{definition}

In the rest of this subsection, we mimic the formal structure found on the groups $F^{p,q}$ and reduce the statement $\cP_{(p,q)}$ to the case $(p,q)=(0,1)$ which is treated in the next section.

\begin{proposition}\label{module-structure-lgroup}
There is a natural bilinear map $L^{p,q} (X,Y) \times \KO^0 (Z,W) \to L^{p,q} ((X,Y) \times (Z,W))$. The induced product $L^{p,q} (X,Y) \times \KO^0 (X,Y) \to L^{p,q}(X,Y)$ turns $L^{p,q}(X,Y)$ into a right-$\KO^0 (X,Y)$-module. The maps $\susp$ and $\Lambda$ are $\KO^0$-linear.
\end{proposition}

\begin{proposition}\label{morita-equivalence-lgroup}
There are natural Morita equivalence isomorphisms $\mor^{1,1}:L^{p,q} (X,Y) \cong L^{p+1,q+1}(X,Y)$ such that the diagrams
\[
\xymatrix{
L^{p,q}(X,Y) \ar[r]^-{\mor} \ar[d]^{\susp} & L^{p+1,q+1}(X,Y) \ar[d]^{\susp}  & & L^{p,q} (X,Y)\ar[r]^-{\mor} \ar[d]^{\Lambda} & L^{p+1,q+1} (X,Y)\ar[d]^{\Lambda}\\
 F^{p+1,q}(X,Y) \ar[r]^-{\mor} &  F^{p+2,q+1}(X,Y) & &  \Omega F^{p,q}(X,Y) \ar[r]^-{\mor} & \Omega F^{p+1,q+1}(X,Y)
}
\]
commute. There are analogous Morita equivalences $\mor^{0,8}$ and $\mor^{8,0}$.
\end{proposition}

The proofs are straightforward adaptions of those for the groups $F^{p,q}$ and are left to the reader. Since the horizontal maps in Proposition \ref{morita-equivalence-lgroup} are isomorphisms and since the Bott map on the $F$-groups is compatible with Morita equivalences, we immediately conclude:

\begin{corollary}\label{morita-mainthm}
The equivalences $\cP_{(p+1,q+1)}\Leftrightarrow \cP_{(p,q)}$, $\cP_{(p,q)}\Leftrightarrow \cP_{(p+8,q)}$ and $\cP_{(p,q)}\Leftrightarrow \cP_{(p,q+8)}$ hold.
\end{corollary}

We now turn to Bott periodicity on the $L^{p,q}$-groups. One defines $\Omega L^{p,q}$ in an analogous fashion as in the case of the group $F^{p,q}$. 
\begin{definition}\label{def-bott}
Let $(V,A)$ be a pair for $L^{p,q}(X,Y)$, $p>0$. Consider the operator $J=e_1 \iota$, acting on $V$. We get a new pair on $ X \times \bI \times M$, by taking the product of the bundle $V$ with $\bI$ and by taking $\bott (A)_{(x,s,t)}=A_{(x,t)} + sJ$. For $s \neq 0$, this is invertible and if $A_{(x,t)}$ was invertible, then so is $A_{(x,t)} + sJ$ (the same calculation as for the Bott map on the $F$-groups).
This respects the equivalence relations in the group $L^{p,q}$ and thus defines a natural homomorphism $\bott: L^{p,q}(X,Y) \to \Omega L^{p-1,q}(X,Y)$.
\end{definition}

It is immediately verified that Bott periodicity commutes with Morita equivalences. The compatibility with the maps $\susp$ and $\Lambda$ takes some work.

\begin{proposition}\label{bott-map-L-groups}
For $p \geq  1$, the following two diagrams commute:
\begin{equation}\label{diagrams}
\xymatrix{
L^{p,q} (X,Y)\ar[r]^-{\bott} \ar[d]^{\Lambda} & \Omega L^{p-1,q}(X,Y) \ar[d]^{\Omega \Lambda} & & L^{p,q} (X,Y)\ar[r]^-{\bott} \ar[d]^{\bott \circ \susp} & \Omega L^{p-1,q}(X,Y) \ar[d]^{\Omega(\bott \circ \susp)}\\ 
\Omega F^{p,q}(X,Y) \ar[r]^-{\bott} & \Omega^2 F^{p-1,q} (X,Y)& & \Omega F^{p,q}(X,Y) \ar[r]^-{\bott} & \Omega^2 F^{p-1,q} (X,Y).
}
\end{equation}
\end{proposition}
\begin{proof}
We apply Lemma \ref{nullhomotopy-criterion} in both cases, and the argument for the case $(X,Y)=(*,\emptyset)$ and the general case are the same, up to change of notation.
Let $A(t)$ be a family of pseudo-Dirac operators representing an element in $L^{p,q}(*)$. The images under $(\Omega\Lambda) \circ \bott$ and $\bott \circ \Lambda$ are represented by the $2$-parameter families
\begin{equation}\label{operators1}
(s,t) \mapsto F (s,t):=\frac{A(t) + sJ}{(1 + (A(t) + sJ)^2 )^{1/2}} \; \text{and} \; (s,t) \mapsto G (s,t):=\normalize{A(t)} + sJ.
\end{equation}

According to Lemma \ref{nullhomotopy-criterion}, it is enough to prove that the anticommutator $\{F(s,t),G(s,t)\}$ is a positive operator. Since $A(t)$ and $J$ anticommute, $(A(t)+sJ)^2 = A(t)^2 + s^2$. Moreover, $J$ commutes with $P=(1+A(t)^2)^{-1/2}$ and $Q(s,t)=(1+(A(t)+sJ)^2)^{-1/2}$. The anticommutator is equal to 
\[
\{A(t)Q(s,t),\normalize{A(t)}\} + \{A(t) Q(s,t),sJ\} + \{sJ Q(s,t),\normalize{A(t)}\} + \{sJ Q(s,t),sJ\}.
\]
The second and third term vanish. The first is $2 A(t)^2 Q(s,t) P \geq 0$, and the fourth one is $2 s^2 Q(s,t)\geq 0$. Thus Lemma \ref{nullhomotopy-criterion} applies, and this proves the commutativity of the left diagram.

For the right diagram, the computation is very similar. Let
\[
D=\twomatrix{A(t)}{-\partial_t}{\partial_t}{-A(t)}, \; L:= \twomatrix{}{-1}{-1}{}, \;K:= \twomatrix{\iota e_1}{}{}{-\iota e_1}.
\]
The images under $\bott \circ \bott\circ \susp$ and $(\Omega (\bott \circ \susp)) \circ \bott$ are represented by the $2$-parameter families
\[
(s,u) \mapsto \normalize{D}+ sL + uK \; \text{and} \; (s,u) \mapsto \normalize{D+uK} + sL,
\]
whose anticommutator is 
\[
2 D^2 \frac{1}{(1+D^2)^{1/2}}  \frac{1}{(1+u^2 + D^2)^{1/2}} + 2 u^2 \frac{1}{(1+u^2 + D^2)^{1/2}} + 2s^2 \geq 0.\qedhere
\]
\end{proof}

\begin{corollary}
The implication $\cP_{(0,1)} \Rightarrow \cP_{(p,q)}$ for all $(p,q)$ holds. In particular, to prove Theorem \ref{main-index-theorem}, it suffices to consider the case $(p,q)=(0,1)$.
\end{corollary}
\begin{proof}
By the Bott periodicity theorem, the bottom maps in the diagrams (\ref{diagrams}) are isomorphisms. Therefore, the implication $\cP_{(p,q)} \Rightarrow \cP_{(p+1,q)}$ holds for all $p\geq 0$.
Assume that $\cP_{(0,1)}$ holds. By Corollary \ref{morita-mainthm}, $\cP_{(p,p+1)}$ follows for all $p \geq 0$. Combining both observations, $\cP_{(p,q)}$ follows for all $q \geq 1$, $p \geq q-1$. Using the implication $\cP_{(p+8,q)} \Rightarrow \cP_{(p,q)}$, $\cP_{(p,q)}$ follows for all $q \geq 1$. Finally $\cP_{(p+1,1)} \Rightarrow \cP_{(p,0)}$. 
\end{proof}

\subsection{A simple special case and conclusion of the argument}\label{conclusion}

What is left to complete the proof of Theorem \ref{main-index-theorem} is the statement $\cP_{(0,1)}$. 
This will be done in two steps. First, we prove $\cP_{(0,1)} (*)$ and then we use Theorem \ref{isomorphism-theorem} and $\KO$-linearity to conclude $\cP_{(0,1)}(X,Y)$ for all pairs $(X,Y)$. Consider the diagram

\begin{equation}\label{normalization-diagram}
\xymatrix{
L^{0,1}(*) \ar[dr]^{\cong}_{\Lambda} \ar[r]^{\susp} & F^{1,1}(*) \ar[r]^{\mor^{-1}}_{\cong}  \ar[d]^{\bott}_{\cong} & F^{0,0} (*)\ar[r]_-{\cong}^-{\ind} & \bZ\\
 & \Omega F^{0,1} (*).\\
}
\end{equation}

The isomorphism $\ind$ is defined as follows. Let $(H,\iota,F)$ be a $(0,0)$-cycle over $*$. We can write 
\[
\iota = \twomatrix{1}{}{}{-1}, \; F = \twomatrix{}{F^*_0}{F_0}{}
\]
where $F_0$ is a Fredholm operator (without further conditions) and we define $\ind [H,\iota,F]$ to be the usual Fredholm index of $F_0$.

\begin{proposition}\label{index-theorem-on-point}
The diagram (\ref{normalization-diagram}) commutes. Hence $\cP_{(0,1)} (*)$ holds.
\end{proposition}
\begin{proof}
The maps $\ind$, $\mor$ and $\bott$ are isomorphisms. By Theorem \ref{isomorphism-theorem}, the map $\Lambda$ is also an isomorphism, so that all groups in (\ref{normalization-diagram}) are isomorphic to $\bZ$. Therefore, it is enough to construct an element $u \in L^{0,1}(*)$ such that $\Lambda (u) = \bott (\susp (u))$ and $\Lambda (u) \neq 0$.

Let $V_0 \to M$ be a real vector bundle (without grading or Clifford structure) and $B$ be a pseudo-Dirac operator on $V_0$. Assume that the kernel of $B$ is $1$-dimensional; such an operator is easily constructed by adding a suitable finite rank operator to an arbitrary pseudo-Dirac operator. Let $p$ be the projection onto $\ker (B)$ and let $a: \bR \to \bR$ be a smooth function with $a(t)=-1$ for $t \leq -1$ and $a(t) =-1$ for $t \geq 1$. 
Then $V= V_0 \oplus V_0$ has the $\Cl^{0,1}$-structure
\[
\iota=\twomatrix{1}{}{}{-1}, \; \eps=\twomatrix{}{1}{1}{} \; \text{ and the curve } \; A(t)= \twomatrix{}{B+a(t)p}{B+a(t)p}{}
\]
represents an element $u \in L^{0,1}(*)$. We now show

\begin{enumerate}
\item $\Lambda (u) = \bott \circ \mor \circ \ind^{-1} (1)$,
\item $\ind \circ \mor^{-1} \circ \susp (u) =1$,
\end{enumerate}
and this finishes the proof. Let us first compute $\bott \circ \mor \circ \ind^{-1} (1) \in F^{1,1}(*)$. The class $\ind^{-1}(1) \in F^{0,0}$ is represented by the one-dimensional space $\bR$, with grading operator $+1$ and $F=0$. Under Morita equivalence, this becomes the element defined by
\[
 \iota:=\twomatrix{1}{}{}{-1}, \; e:= \twomatrix{}{-1}{1}{}, \; \eps:=\twomatrix{}{1}{1}{}, \; F=0
\]
in $F^{1,1}(*)$. The Bott path on this element is 
\[
t  \mapsto t  e\iota = \twomatrix{}{t}{t}{} \sim \twomatrix{}{a(t)}{a(t)}{}.
\]

Let $H= L^2 (M;V)$ and let $K \subset H$ be the kernel of $A(0)$ with orthogonal complement $K^{\bot}$. Then $\Lambda (u)$ is represented by $(K, \normalize{A(t)}|_K) \oplus (K^{\bot}, \normalize{A(t)}|_{K^{\bot}})$. The second summand is acyclic, and the first summand is a representative for $\bott \circ \mor \circ \ind^{-1} (1)$, by the above computation. This completes the computation of $\Lambda(u)$. 

The computation of $\mor^{-1} (\susp (u)) \in F^{0,0}(*)$ is straightforward, using the formulas for the Morita equivalence and the suspension. We have to restrict the operator $\twomatrix{A(t)}{-\partial_t}{\partial_t}{-A(t)}$ to the $+1$-eigenspace of $\twomatrix{\eps}{}{}{\eps}\twomatrix{}{\iota}{-\iota}{}$. This eigenspace is isomorphic to $L^2 (\bR \times M; V)$; the grading on the suspension and the operator restrict to

\begin{equation}\label{exampleoperator}
\iota=\twomatrix{1}{}{}{-1} \;\text{and} \; C:=\twomatrix{}{-\partial_t+B+a(t)p}{\partial_t+B+a(t)p}{}.
\end{equation}
To verify that $\ind \circ \mor^{-1} \circ \susp (u) =1$, we have to show that
\[
 \ind (\partial_t+B+a(t)p) = \dim (\ker (\partial_t+B+a(t)p))- \dim (\ker(-\partial_t+B+a(t)p)) =1.
\]
To compute these kernels, observe that for $u: \bR \to \Gamma (M;V)$, we have
\[
\norm{Cu}^2 = \norm{\twomatrix{}{B}{B}{} u}^2 + \norm{ \twomatrix{}{-\partial + ap}{ \partial+ ap}{} u}^2
\]
and so the kernel is precisely the space of $L^2 $-functions $u: \bR \to \ker \twomatrix{}{B}{B}{}$ that satisfy the ODE
\[
\twomatrix{}{-\partial_t+a(t)}{\partial_t +a(t)}{} u (t)=0.
\]
But it is easy to see that the $L^2$-kernel of $\partial_t+a(t)$ is $1$-dimensional, and the $L^2$-kernel of $-\partial_t +a(t)$ is zero.
\end{proof}
The index of the operator (\ref{exampleoperator}) can alternatively be computed using the spectral-flow index theorem as proven by Robbin and Salamon \cite[Theorem 4.21]{RoSa}.

\begin{corollary}\label{normalform-l-01}
Theorem \ref{main-index-theorem} holds for $(p,q)=(0,1)$.
\end{corollary}
\begin{proof}
Let $x \in L^{0,1} (X,Y)$ and let $u \in L^{0,1}(*)$ be the element constructed in the proof of 
Proposition \ref{index-theorem-on-point}.
Let $a$ be the image of $x$ under the composition 
\[
L^{0,1} (X,Y) \stackrel{\Lambda}{\to} \Omega F^{0,1} (X,Y) \stackrel{\bott^{-1}}{\to } F^{1,1} (X,Y) \cong \KO^0 (X,Y)
\]
of $\KO^0$-linear isomorphisms. Since the element $u \in L^{0,1}(*)$ maps to $1$ under these isomorphisms for $(X,Y)=*$ by the proof of Proposition \ref{index-theorem-on-point}, it follows that $x = u \cdot a$. Then
\[
\bott (\susp (x)) = \bott (\susp (u \cdot a)) =   \bott (\susp ( u)) \cdot a = \Lambda (u) a  = \Lambda (u \cdot a) = \Lambda (x);
\]
the first and last equations are clear, the third is Proposition \ref{index-theorem-on-point} and the two others follow from $\KO$-linearity.
\end{proof}

\appendix

\section{Comparison of the models for \texorpdfstring{$K$}--Theory}\label{appendix}

\subsection{Comparison for compact pairs}

In this subsection, we prove Theorem \ref{compact-comparison-theorem}. Recall the statement. For each space pair $(X,Y)$, there is a map
\[
\alpha:[(X,Y);(\Fred^{p,q} (H),\cG^{p,q} (H))] \to F^{p,q} (X,Y).
\]
We have to show that $\alpha$ is bijective whenever $(X,Y)$ is a compact pair.
The proof can be summarized in the commutative diagram

\[
\xymatrix{
 [(X,Y) ;(\Fred^{p,q} (H),\cG^{p,q} (H))]\ar[r]^-{\alpha}  & F^{p,q}(X,Y) & \\
[(X,Y) ;(\Fred^{p,q}_{0} (H),\cG^{p,q}_{0} (H))] \ar[r]^-{\alpha_{0}} \ar[u] & F^{p,q}_{0}(X,Y) \ar[u] \ar[r]^-{\gamma} & \KK (\Cl^{q,p} ;\gR_0(X-Y)).
}
\]

The groups in this diagram have the following meaning:

\begin{enumerate}
\item The space $\Fred^{p,q}_{0} (H) \subset \Fred^{p,q} (H)$ is the subspace of operators $F$ such that $F^2 -1$ is compact. The space $\cG^{p,q}_{0} (H)) \subset \cG^{p,q} (H))$ is the space of involutions $F$. The left upwards map is induced by forgetting the stronger conditions and is an isomorphism by a spectral-deformation argument as in \cite{AS69}.
\item Similarly, the group $F^{p,q}_{0} (X,Y)$ is defined in the same way as $F^{p,q}(X,Y)$, using Fredholm families, except that we impose the stronger conditions that $F^2-1$ is compact and $F_y^2-1 =0$ for $y \in Y$. Lemma \ref{lemmaa} below proves that the middle upwards map is an isomorphism.
\item $\KK (\Cl^{q,p} ;\gR_0(X-Y))$ is Kasparov's real $\KK$-group (note the switch in the Clifford degrees). By $\gR_0(X-Y) $, we denote the real $\cstar$-algebra of real valued continuous functions on $X$ that vanish on $Y$.
We will define the map $\gamma$ below. The composition $\gamma \circ \alpha_{0}$ is an isomorphism by a classical result of Kasparov \cite[\S 6]{Kasp}. Thus the proof of Theorem \ref{compact-comparison-theorem} will be complete once we show that the map $\gamma$ is injective.
\end{enumerate}

\begin{lemma}\label{lemmaa}
The map $F^{p,q}_{0} (X,Y) \to F^{p,q} (X,Y)$ is an isomorphism.
\end{lemma}
\begin{proof}
First we prove surjectivity. Let $(H,F)$ be a cycle for $F^{p,q} (X,Y)$. There exists $\epsilon >0$ such that $F_{y}^{2} \geq \epsilon^2$ for all $y \in Y$ and $F^{2} \geq \epsilon^2 \bmod \Kom_X (H)$, by compactness of $X$. Let $f : \bR \to \bR$ be an odd function such that $f(t)=1$ for $t \geq \epsilon$. Using functional calculus, we define $f(F) \in \Lin_X (H)$. The pair $(H, f(F))$ is a cycle for $F^{p,q}_{0} (X,Y)$. The straight-line homotopy between $(H,f(F))$ and $(H,F)$ proves that $[H,F] = [H, f(F)] \in F^{p,q} (X,Y)$, which proves surjectivity. Injectivity is proven by the same argument, applied to a concordance.
\end{proof}

To proceed, the reader should recall the definition of Kasparov's $\KK$-Theory \cite{Bla,Kasp}. In particular, familiarity with the basic notions of the theory of Hilbert-$\cstar$-modules will be assumed \cite[\S 15]{WO}. Here is the definition of the map $\gamma$.
Let $(V,\iota,c,F)$ be a $(p,q)$-cycle on $(X,Y)$ and assume that $F^2 -1$ is compact and $F_y^2 -1 =0$ for $y \in Y$. The space $\Gamma_c (X-Y;V)$ of sections vanishing at infinity is a graded Hilbert $\gR_0 (X-Y)$-module. The operator $F$ is an odd self-adjoint operator on $\Gamma_c (X-Y; V)$, and we claim that $F^2-1: \Gamma_c (X-Y;V)\to \Gamma_c (X-Y;V)$ is compact in the Hilbert module sense. This follows from general principles: given $\eps>0$, we pick $G \in \Fin_X (V)$ with $\norm{G_x-(F^2_x-1)} \leq \eps/2$ and a function $a \in C_c (X-Y, \bR)$ with $\norm{aG_x-G_x} \leq \eps/2$, for all $x \in X$. Then $aG: \Gamma_c (X-Y;V)\to \Gamma_c (X-Y;V)$ is of finite rank, and it follows that $F^2-1$ is compact.

Using the $\Cl^{p,q}$-structure, we define a graded $*$-homomorphism $\rho:\Cl^{q,p} \to \Lin_{\gR(X)} (\Gamma_c (X-Y;V))$ (sic) into the $\cstar$-algebra of adjointable operators of the Hilbert-$*$-module by $\rho(e_i):=c(\eps_i \iota )$ and $\rho(\eps_j):= c(e_j \iota)$. Since we assumed that $F$ is $\Cl^{p,q}$-linear and odd, we find that for all $z \in \Cl^{q,p}$, the graded commutator $[\rho(z),F]=0$ is trivial. Thus $(\Gamma_c (X-Y; V), \rho,F)$ is a Kasparov $\Cl^{q,p}-\gR_0 (X-Y)$-module and defines an element in $\KK (\Cl^{q,p};\gR_0 (X-Y))$. We define

\[
\gamma (V,\iota,c,F) := [\Gamma_c (X-Y; V), \rho,F] \in \KK (\Cl^{q,p};\gR_0 (X-Y)).
\]

If $F^2-1=0$ on all of $X$, then $(\Gamma_c (X-Y; H), \rho,F)$ is degenerate and hence the $\KK$-class is zero, so that the above construction is compatible with the equivalence relations defining $F^{p,q}_{0} (X,Y)$ and $\KK(\Cl^{q,p};\gR_0 (X-Y))$ and so $\gamma$ is well-defined.
It was proven by Kasparov \cite[\S 6]{Kasp} that the composition $\gamma \circ \alpha_{0}
$ is an isomorphism.

\begin{lemma}\label{lemmab}
The map $\gamma:F^{p,q}_{0} (X,Y) \to \KK(\Cl^{q,p} ;\gR_0(X-Y))$ is injective.
\end{lemma}
\begin{proof}
Let $[V,F] \in F^{p,q}_{0}(X,Y)$ such that $\gamma [V,F] =0 \in \KK (\Cl^{q,p};\gR_0(X-Y))$. By definition of the $\KK$-group, this means that the Kasparov module $(\Gamma_0 (X-Y; V), \rho, F)$ is homotopic to the zero module (in the sense of \cite[Definition 17.2.2]{Bla}). Let $(E, \phi,G)$ be a nullhomotopy, this is, by definition, a Kasparov $\Cl^{q,p}$-$\gR_0([0,1] \times (X-Y ))$-module which restricts to $0$ in $\gR_0(\{1\} \times (X-Y ))$ and to $(\Gamma_0 (X-Y; V), \rho, F)$ in $\gR_0(\{0\} \times (X-Y ))$.
Let $H$ be an ample $\Cl^{p,q}$-Hilbert space. Consider the graded Hilbert $\gR_0([0,1] \times (X-Y ))$-module $E':=\Gamma_0 ([0,1] \times (X-Y ); H_{[0,1] \times (X-Y)})$.
Let $\xi: \Cl^{q,p} \to \Lin_{\gR_0([0,1] \times (X-Y ))} (E')$ be the associated representation as constructed above. Any $J \in  \cG^{p,q}_{0}(H)$ induces an operator in $E'$. The Kasparov $\Cl^{q,p}$-$\gR_0([0,1] \times (X-Y ))$-module $(E', \xi, J)$ is degenerate, hence nullhomotopic by \cite[Proposition 17.2.3]{Bla} and so zero in the $\KK$-group. Form the direct sum $(E, \phi,G) \oplus (E', \xi, J)$ of the nullhomotopy with this degenerate module.
By Kasparov's stabilization theorem \cite[Theorem 15.4.6]{WO}, the Hilbert module $E \oplus E'$ is isomorphic to the space of continuous function $X-Y \to H$ vanishing at infinity. Therefore, $(E, \phi,G) \oplus (E', \xi, J)$ provides a concordance of $(V,F) \oplus (H_{X-Y},J)$ to an invertible Fredholm family. 
This shows that $[H,F] =0 \in F^{p,q}(X,Y)$.
\end{proof}

\subsection{Representability of the \texorpdfstring{$F^{p,q}$}--functors}

Here we construct the representing space for the functor $F^{p,q}$, following ideas of Atiyah and Segal \cite{Atiseg}.

\begin{definition}
Let $H$ be an ample $\Cl^{p,q}$-Hilbert space. We define $K^{p,q}$ as the space of pairs $(F,G)$ of bounded $\Cl^{p,q}$-Fredholm operators on $H$, with the condition that $FG-1$ and $GF-1$ are both compact. The topology on $K^{p,q}$ is induced from the injective map
\[
K^{p,q} \to \Lin(H)_{c.o.} \times \Lin(H)_{c.o.} \times \Kom (H) \times \Kom (H), \; (F,G) \mapsto (F,G,FG-1,GF-1).
\]
The subspace $D^{p,q} \subset K^{p,q}$ is the subspace of pairs $(F,G)$ such that $FG = GF = 1$.
\end{definition}

To understand the rationale of this definition, note that the possible parametrices to a Fredholm operator form a convex set, the bookkeeping of $G$ and $FG-1$, $GF-1$ is only there to guarantee the correct continuity condition.

If $(X,Y)$ is a space pair, $X$ compactly generated, and $F: (X,Y) \to (K^{p,q},D^{p,q})$, we get a Fredholm family $(F(x))_{x \in X}$ on the trivial Hilbert bundle $H_X \to X$. Therefore, we obtain a map
\begin{equation}\label{atiyah-singer-karoubi4}
[(X,Y); (K^{p,q},D^{p,q})] \to F^{p,q} (X,Y)
\end{equation}
and Theorem \ref{paracompact-comparison-theorem} asserts that (\ref{atiyah-singer-karoubi4}) is bijective whenever $(X,Y)$ is a paracompact pair.
For the proof, we use a trick that we propose to call ``Dixmier-Douady swindle'', because it first appears in \cite{DD}. The trick itself is stated as Lemma \ref{dixmiduedue} below. We also need a consequence that was drawn in \cite{DD}.

\begin{theorem}\label{absorbtion-lemma} \cite[Th\'eor\`eme 4]{DD}
Let $X$ be a paracompact space and let $V \to X$ be a $\Cl^{p,q}$-Hilbert bundle. Let $H$ be an ample $\Cl^{p,q}$-Hilbert space. Then there is an isomorphism $V \oplus H_X \cong H_X$ of $\Cl^{p,q}$-Hilbert bundles.
\end{theorem}

In \cite{DD}, only the existence of an isometry of real Hilbert bundles is proven. To upgrade this to $\Cl^{p,q}$-Hilbert bundles, one uses Lemma \ref{kuipertheoremb}.

Let us prove surjectivity of (\ref{atiyah-singer-karoubi4}). Let $(V,F)$ be a $(p,q)$-cycle. Let $(H,J)$ be an ample $\Cl^{p,q}$-Hilbert space and $J \in \cG^{p,q}_{0}(H)$. By Theorem \ref{absorbtion-lemma}, there is an isometry of $\Cl^{p,q}$-Hilbert bundles $V \oplus H_X \cong H_X$.
In the group $F^{p,q}(X,Y)$, the equation $[V,F] = [V \oplus H_X,F \oplus J]$ holds. But the right-hand side lies in the image of (\ref{atiyah-singer-karoubi4}).

For the injectivity, we explain the Dixmier-Douady swindle. Let $N_0$ be a $\Cl^{p,q}$-module which contains each irreducible one and form $N:=N_0 \oplus N^{op}_0$. Then $H:= L^2 ([0,1]; N)$ is an ample $\Cl^{p,q}$-Hilbert space.
Let $J =\twomatrix{}{1}{1}{}$, a $\Cl^{p,q}$-linear odd involution on $N \oplus N^{op}$. It induces a self-adjoint isometry $J: H \to H$.
The basis for the trick is the following result.

\begin{lemma}\label{dixmiduedue}\cite[Lemme 2]{DD}
Let $H_t \subset H$ be the subspace of all functions that are supported in $[0,t]$. Let $P_t: H \to H_t$ be the orthogonal projection. Then $t \mapsto P_t$ is a continuous map $[0,1] \to \Lin (H)_{c.o.}$. 
Moreover there exist isometries $Q_t: H_t \cong H$, for $t \in (0,1]$, such that $Q_1 =1$ and such that the maps $(0,1] \to \Lin(H)_{c.o.}$, $t \mapsto Q_{t}^{-1}$, $t \mapsto Q_t P_t$ are continuous.
Finally, the adjoint of $Q_t P_t$ is $Q_{t}^{-1}$.
\end{lemma}

Actually, Dixmier and Douady prove this when $\Lin (H)_{c.o.}$ is replaced by $\Lin (H)_{stop}$, the space of bounded operators with the strong operator topology. By the remarks on page 38 \cite{Atiseg}, one gets continuity with target $\Lin (H)_{c.o.}$. Moreover, Dixmier and Douady do not mention Clifford algebras, but the necessary modifications are easy to do. The last statement is clear from the construction.
Using the functions from Lemma \ref{dixmiduedue}, we define maps 
\begin{equation}\label{dixdouswbou}
 (0,1] \times \Lin (H)_{c.o.} \to \Lin (H)_{c.o.}, \, (t,A) \mapsto J (1-P_t) + Q_{t}^{-1} A Q_t P_t
\end{equation}
and 
\begin{equation}\label{dixdouswcom}
 (0,1] \times \Kom (H) \to \Kom (H), \, (t,A) \mapsto  Q_{t}^{-1} A Q_t P_t.
\end{equation}
These maps preserve Fredholm operators, self-adjoint operators, even/odd operators and $\Cl^{p,q}$-linear operators.

\begin{lemma}
The maps (\ref{dixdouswbou}) and (\ref{dixdouswcom}) are continuous.
\end{lemma}
\begin{proof}
The map (\ref{dixdouswbou}) is the composition
\[
(0,1] \times \Lin (H)_{c.o.} \to (0,1] \times (\map((0,1]; \Lin (H)_{c.o.}))^3 \times \Lin (H)_{c.o.} 
\to \Lin (H)_{c.o.},
\]
the first map is $(t,A) \mapsto (t,P, Q^{-1}, Q P,A)$ and is obviously continuous. The second map is $(t,X,Y,Z,A) \mapsto J (1-X(t))+ Y(t) A Z(t)$, and this is continuous because the evaluation maps are continuous, because the composition maps between mapping spaces are continuous (and because $H$ is compactly generated), and because $\Lin (H)_{c.o.}$ is a topological vector space.

For the map (\ref{dixdouswcom}), it is enough to prove sequential continuity. Let $K_n \to K \in \Kom (H)$ (in norm), and $t_n \to t >0 $ in $(0,1]$. Let $P_n:= P_{t_n}$, $P:=P_t$ etc. By Lemma \ref{dixmiduedue}, $Q_n P_n \to Q P$ and $(Q_n P_n)^* = Q_{n}^{-1} \to Q^{-1} = (QP)^*$, both in the compact-open topology. Continuity of the map (\ref{dixdouswcom}) follows from the following general observation: Assume that $K_n \to K$ is a norm convergent sequence of compact operators and that $A_n$ and $A$ are bounded operators such that $A_n \to A$ and $A_{n}^{*} \to A^*$ in the compact-open topology. Then $A_n K_n A_{n}^{*} \to AKA^*$, in the norm topology. 

For the proof of this claim, we first observe that, by the Banach-Steinhaus theorem, $\norm{A_n} , \norm{A} \leq C$ for some $C$. So
\[
\norm{A_n K_n - AK } \leq \norm{A_n K_n - A_n K} + \norm{A_n K - A K} \leq C \norm{K_n-K} + \norm{A_n K - A K}.
\]
Since $A_n \to A$ in the compact-open topology and since $K$ is compact, $\norm{A_n K - AK} \to 0$. Abbreviate $B_n= A_n K_n$ and $B=AK$. We want to show that $B_n A_{n}^{*} \to BA^*$, in norm. But 
\[
\norm{B_n A_{n}^{*} - BA^*} = \norm{A_n B^*_n -AB^* }.
\]
By the first part of the argument, $B_{n}^{*} \to B^*$ in the norm topology. Again by the first part of the argument, it follows that $A_n K_n A_{n}^{*} \to AKA^*$ in the norm topology, as claimed.
\end{proof}

\begin{proposition}\label{dixdueprop}
Let $H$ be an ample $\Cl^{p,q}$-Hilbert space. Then there exists a decomposition $H= H_0 \oplus H_1$ into ample $\Cl^{p,q}$-Hilbert spaces, $J \in \cG^{p,q}_{0}(H_1)$ and a homotopy $R_t : [1/2,1] \times K^{p,q} \to K^{p,q}$ such that $R_1 = \id$ and such that for all $x \in K^{p,q}$, $R_{1/2} (x)$ is of the form $y \oplus J$, for some $y$. The homotopy preserves the subspace $D^{p,q}$.  
\end{proposition}
\begin{proof}
One uses the homotopies (\ref{dixdouswbou}) and (\ref{dixdouswcom}), for $t \in [1/2,1]$. 
\end{proof}

\begin{proof}[Proof that the map (\ref{atiyah-singer-karoubi4}) is injective]
Let $F_0 , F_1: (X,Y) \to (K^{p,q},D^{p,q})$ be two maps that induce the same element in $F^{p,q}(X,Y)$. By the definition of $F^{p,q}(X,Y)$ and by Theorem \ref{absorbtion-lemma}, this means that $R_{1/2} \circ F_0$ and $R_{1/2} \circ F_1$ are homotopic. Since $R_{1/2} \circ F_i$ and $F_i$ are homotopic by Proposition \ref{dixdueprop}, $F_0 \sim F_1$. The homotopy preserves the subspace $D^{p,q}$, and the proof is complete.
\end{proof}

\begin{proposition}
The space $D^{p,q}$ is weakly contractible.
\end{proposition}
\begin{proof}
Inside $D^{p,q}$, there is the subspace $E^{p,q}$ of orthogonal elements. That $E^{p,q}$ is contractible is proven using the Dixmier-Douady swindle, in the same way as \cite[Proposition A2.1]{Atiseg}.
It remains to prove that $E^{p,q} \to D^{p,q}$ is a weak homotopy equivalence. Let $X$ be a compact space and $F: X \to D^{p,q}$ a map (there is no need to consider a relative map). The family $F: H_X \to H_X$, $(x,v) \mapsto (x, F(x)v)$ is an invertible homomorphism. There exists $\epsilon >0$ such that $F(x)^2 \geq \epsilon^2$. Let $f: \bR \to \bR$ be an odd continuous function such that $f(t)=1$ for $|t| \geq \epsilon$. The family over $I \times X$, $(s,x) \mapsto (1-s) F(x) + s f(F(x))$ is a homotopy from $F$ to a map into $E^{p,q}$.
\end{proof}

\section{A fiber theorem for spaces of units in \texorpdfstring{$C^*$}--algebras}

Let $A$ and $B$ be unital $\bZ/2$-graded $\cstar$-algebras and $\pi:A \to B$ be a graded surjective $*$-algebra homomorphism. 
Let $A_{sa,odd}$ and $ B_{sa,odd}$ be the subspaces of odd self-adjoint elements. Let $C \subset B_{sa,odd}$ be the (open) subset of invertible elements, $F:= \pi^{-1} (C) \cap A_{sa,odd}$ and $G \subset F$ be the subspace of invertible elements.

\begin{proposition}\label{fibretheorem}
The restriction of $\pi$ to $F \to C$ and to $G \to C$ are Serre fibrations, and the first one even admits a global section.
\end{proposition}

If we ignore the conditions (self-adjoint, odd), this is a well-known result (see \cite[Theorem 11]{Pal2}). However, we do not see how to derive it from the classical one directly, and therefore we prove Proposition \ref{fibretheorem} from scratch.

One example where this result can be applied is when $H$ is a $\Cl^{p,q}$-Hilbert space, and $A$ is the algebra of all Clifford-linear bounded operators on $H$ (no condition on the grading). The grading of $A$ is by even/odd operators. In that case, we let $B$ be the Calkin algebra, i.e. the quotient of $A$ by the ideal of compact operators.

\begin{lemma}\label{fibrelemma}
Let $y: I \to C$ be a continuous path and $x_0 \in G$ with $\pi (x_0)=y(0)$. Then there exists a path $x: I \to G$ with $x(0)=x_0$ and $\pi \circ x=y$.
\end{lemma}
\begin{proof}
It is a theorem by Bartle and Graves that if $\phi:X \to Y$ is a surjective bounded operator between Banach spaces, then there is a continuous cross-section $\sigma:Y \to X$ (\cite{BartGrav}, see \cite[p. 187]{Holm} for an explicit statement). 
Choose a section $\sigma: B_{sa,odd} \to A_{sa,odd}$ with $\sigma (0)=0$. 

Since $x_0$ and $y(t)$ are invertible, there exists $\epsilon>0$ such that
$[-2\epsilon, 2\epsilon] \cap \spec_B (y(t))= [-2\epsilon, 2\epsilon] \cap \spec_A (x_0)=\emptyset$ for all $t \in I$. Choose $\delta>0$ with $\norm{\sigma (z)} < \epsilon$ for all $\norm{z}< \delta$. Since $y$ is uniformly continuous, there is $r \in \bN$ such that $|t-t'| \leq \frac{1}{r}$ implies $\norm{y(t)-y(t')} < \delta$.
Choose an odd continuous function $f: \bR \to \bR$ such that $f (s)=s$ for $s \geq 2\epsilon$ and $f(s)=2 \epsilon$ for $\epsilon \leq s \leq 2 \epsilon$. 
We wish to construct the lift $x$ in such a way that $\spec_A (x(t)) \cap [-2\epsilon, 2\epsilon] = \emptyset$ and we do this inductively on the intervals $[0, \frac{i}{r}]$, $i=0, \ldots, r$. By the choice of $\epsilon$, the given initial value $x_0$ has this property. Now suppose such a lift has been constructed on the interval $[0, \frac{i}{r}]$, for some $i \geq 0$.
For $t \in [\frac{i}{r}, \frac{i+1}{r}]$, define
\[
\tilde{x} (t):= x(\frac{i}{r}) + \sigma (y(t)-y(\frac{i}{r})).
\]
It is clear that $\pi ( \tilde{x}(t))=y(t)$ and because $\sigma (0)=0$, it is also clear that $\tilde{x}(\frac{i}{r})= x(\frac{i}{r})$.
Since $\norm{\sigma (y(t)-y(\frac{i}{r}))} < \epsilon$, it follows that $\spec_A (\tilde{x}(t)) \cap [-\epsilon, \epsilon] = \emptyset$. By the choice of the function $f$, $f(\tilde{x}(t))$ has spectrum outside $[-2\epsilon, 2\epsilon]$, and is in $A_{sa,odd}$, because $f$ is an odd function. Also, the choice of $f$ implies that $f(\tilde{x}(\frac{i}{r}))= x(\frac{i}{r})$. Set $x(t):=f(\tilde{x}(t))$. It remains to prove that $\pi(x(t))=y(t)$. The function $h(s):= f(s)-s$ has support in $[-2\epsilon,2\epsilon]$. But
\[
\pi(x(t)) = \pi(x(t)-\tilde{x}(t)) + \pi (\tilde{x}(t)) =  \pi(x(t)-\tilde{x}(t)) +y(t)=  \pi (h(\tilde{x}(t)))+y(t).
\] 
Since the functional calculus commutes with algebra homomorphisms, we conclude that 
\[
\pi (h(\tilde{x}(t))) =  h (\pi(\tilde{x}(t))) = h(y(t)) =0; 
\]
the last equality holds because the support of $h$ is disjoint from the spectrum of $y(t)$.
\end{proof}

\begin{proof}[Proof of Proposition \ref{fibretheorem}]
The case of $F \to C$ is covered by \cite[Theorem 10]{Pal2} (or follows quickly from the Bartle-Graves theorem). The other case follows from Lemma \ref{fibrelemma}. Let $X$ be a finite CW complex and consider the $C^*$-algebras $C^0 (X; A)$ and $C^0 (X;B)$. The induced homomorphism $\pi _*: C^0 (X;A) \to C^0 (X;B)$ is surjective (this follows from the Bartle-Graves theorem). Lemma \ref{fibrelemma}, applied to $\pi_*$, shows that the map $C^0 (X;G) \to C^0 (X;C)$ has the path-lifting property. But any lifting problem
\[
\xymatrix{
X \times \{0\} \ar[r]^{g} \ar[d] & G  \ar[d]^{\pi}\\
X \times [0,1] \ar[r]^{f} \ar@{..>}[ur]  & C
}
\]
is equivalent to a path-lifting problem for $\pi_*$.
\end{proof}

\bibliographystyle{plain}
\bibliography{spectralflowbott}

\end{document}